\documentclass{article}
\usepackage{graphicx} % Required for inserting images
\usepackage{amsmath,amssymb,amsthm}
\usepackage{amsrefs}
\usepackage{bm} %bold 
\usepackage{bbm} %blackboardbold 
\usepackage{color} %color 
\usepackage{soul}
\usepackage{comment}
\usepackage{authblk}
\usepackage{fullpage}

\usepackage[shortlabels]{enumitem}

\allowdisplaybreaks 

\theoremstyle{plain}
\newtheorem{theorem}{Theorem}[section]
\newtheorem{corollary}[theorem]{Corollary}
\newtheorem{proposition}[theorem]{Proposition}
\newtheorem{lemma}[theorem]{Lemma}
\newtheorem*{theorem*}{Theorem}
\newtheorem*{lemma*}{Lemma}
\newtheorem*{proposition*}{Proposition}
\newtheorem*{corollary*}{Corollary}

\theoremstyle{definition}
\newtheorem{definition}[theorem]{Definition}
\newtheorem{example}[theorem]{Example}
\newtheorem{remark}[theorem]{Remark}

\newtheorem*{definition*}{Definition}
\newtheorem*{example*}{Example}
\newtheorem*{remark*}{Remark}

\numberwithin{equation}{section}

\DeclareFontEncoding{FMS}{}{}
\DeclareFontSubstitution{FMS}{futm}{m}{n}
\DeclareFontEncoding{FMX}{}{}
\DeclareFontSubstitution{FMX}{futm}{m}{n}
\DeclareSymbolFont{fouriersymbols}{FMS}{futm}{m}{n}
\DeclareSymbolFont{fourierlargesymbols}{FMX}{futm}{m}{n}
\DeclareMathDelimiter{\VERT}{\mathord}{fouriersymbols}{152}{fourierlargesymbols}{147}

\title{An extension of the Wiener-Wintner ergodic theorem for pointwise jointly ergodic systems and its applications}
\date{\today}

\begin{document}

\author[1]{Michihiro Hirayama\thanks{hirayama@math.tsukuba.ac.jp}}
\affil[1]{Department of Mathematics, University of Tsukuba,  Japan}

\author[2]{Younghwan Son\thanks{yhson@postech.ac.kr}}
\affil[2]{Department of Mathematics,  POSTECH, Korea}

\maketitle

\begin{abstract}
A joint measure-preserving system is $(X, \mathcal{B}, \mu_{1}, \dots, \mu_{k}, T_{1}, \dots, T_{k})$, where each $(X, \mathcal{B}, \mu_{i}, T_{i})$ is a measure-preserving system and any $\mu_{i}$ and $\mu_{j}$ are mutually absolutely continuous probability measures. 
Such a system is called pointwise jointly ergodic if, for any set of bounded measurable functions $f_{1}, \dots, f_{k}$ on $X$, the multilinear ergodic average of their joint action under the transformations $T_{1}, \dots, T_{k}$ converges almost everywhere to the product of their integrals with respect to the corresponding measures.

In this paper, we extend the classical Wiener-Wintner ergodic theorem to the setting of pointwise jointly ergodic systems with nilsequences weight.
Additionally, we provide applications that include results on the mean convergence of weighted ergodic averages and the almost everywhere convergence of ergodic averages taken along subsequences of the form $\lfloor \alpha n \rfloor$, where $\alpha \geq 1$.   
\end{abstract}

\section{Introduction}

The classical Wiener-Wintner ergodic theorem \cite{Wiener-Wintner1941} states that for any measure-preserving system $(X, \mathcal{B}, \mu, T)$ and any $f \in L^{1}(X,\mu)$, there exists a set $X_{f}\subset X$ of full $\mu$-measure such that the weighted ergodic averages
\begin{equation}\label{classical_ww_ergodic_average}
\frac{1}{N} \sum_{n=1}^{N} \mathrm{e}(n\theta) f(T^{n}x)
\end{equation}
converge as $N \to \infty$ for every $x \in X_{f}$ and for every $\theta \in \mathbb{R}$, where $\mathrm{e}(\theta) = \exp{(2 \pi \sqrt{-1} \theta)}$.  
Since the full measure set $X_{f}$ is independent of $\theta$, the Wiener-Wintner theorem is nontrivial strengthening of the Birkhoff ergodic theorem, which corresponds to the special case when $\theta =0$.

%Note that the special case of $\lambda =1$ is the Birkhoff ergodic theorem. The Wiener-Wintner theorem is nontrivial strengthening of the Birkhoff ergodic theorem in the sense that the full measure set $X_0$ is independent of $\lambda$.

There are various possible extensions of this result. 
One natural extension is to generalize the sequence of weights.
Lesigne \cite{Lesigne1990} showed that the weight $\mathrm{e}(n\theta)$ can be replaced with $\mathrm{e}(P(n))$, where $P(n)$ is a real polynomial, and later Host and Kra \cite{Host-Kra2009} generalized the Wiener-Wintner theorem to the class of nilsequences.

There have also been results concerning multilinear Wiener-Winter Theorem. 
Consider the following weighted ergodic averages 
\begin{equation}
\label{eqn:multilinear}
\frac{1}{N} \sum_{n=1}^{N} a_{n} \prod_{i=1}^k f_{i} (T^{in} x)
\end{equation} 
where $(a_{n})_{n \in \mathbb{N}}$ is a sequence and $f_{1}, \dots, f_{k} \in L^{\infty} (X, \mu)$. 
A measure-preserving system $(X, \mathcal{B}, \mu, T)$ is said to satisfy $k$-linear pointwise convergence if, when $a_{n} \equiv 1$, the averages of \eqref{eqn:multilinear} converge almost everywhere for any $f_{1}, \dots, f_{k} \in L^{\infty} (X, \mu)$.
It is a long standing conjecture that any measure-preserving system satisfies $k$-linear pointwise convergence. 
Bourgain established this conjecture for the case $k=2$ in \cite{Bourgain1990}.  

Based on Bourgain's result, Assani and Moore \cite{Assani-Moore2018} demonstrated that for any measure-preserving system $(X, \mathcal{B}, \mu, T)$ and any $f_{1}, f_{2} \in L^{\infty}(X,\mu)$, there exists a set $X_{f_{1}, f_{2}}\subset X$ of full $\mu$-measure such that for any $a_{n} = \mathrm{e}(P(n))$, where $P(n)$ is a real polynomial, the averages of \eqref{eqn:multilinear} converge for all $x \in X_{f_{1}, f_{2}}$. 
Xiao \cite{Xiao2024} and Zorin-Kranich \cite{Zorin-Kranich2015} have shown that if  $(X, \mathcal{B}, \mu, T)$ satisfies $k$-linear pointwise convergence, then for any $f_{1}, \dots, f_{k} \in L^{\infty}(X, \mu)$, there exists a $\mu$-full measure set $ X_{f_{1}, \dots, f_{k}} \subset X$ such that for any nilsequence $(a_{n})_{n \in \mathbb{N}}$, the averages of\eqref{eqn:multilinear} converge for all $x \in X_{f_{1}, \dots, f_{k}}$.

In this paper, we establish an extension of the Wiener-Wintner theorem for pointwise jointly ergodic systems.  
The notion of joint ergodicity of measure-preserving systems was introduced and studied by Berend and Bergerlson in \cites{Berend-Bergelson1984, Berend-Bergelson1986}. 
Measure-preserving transformations $T_{1}, \dots, T_{k}$ on a probability space $(X, \mathcal{B}, \mu)$ are called \emph{$L^{2}$-jointly ergodic} if, for any bounded functions $f_{1}, \dots, f_{k}\in L^{\infty}(X,\mu)$, the multiple ergodic averages $\frac{1}{N} \sum\limits_{n=1}^{N} \prod\limits_{i=1}^{k} f_{i} \circ T_{i}^{n} $ converge to $\prod\limits_{i=1}^{k} \int_{X} f_{i} \, d \mu$ in $L^{2}(X,\mu)$ as $N \to \infty$. 
In \cites{Bergelson-Son2023, Bergelson-Son2022}, the phenomenon of joint ergodicity was explored for the case where the involved transformations $T_{1}, \dots, T_{k}$ may have different invariant measures.

Let $ {\bm X} = (X, \mathcal{B}, \mu_{1}, \dots, \mu_{k}, T_{1}, \dots, T_{k})$ be a joint measure-preserving system, where each ${\bm X}_{i} = (X, \mathcal{B}, \mu_{i}, T_{i})$ is a measure-preserving system on a Lebesgue probability space, and measures $\mu_{i}$ and $\mu_{j}$ are equivalent, that is, they are mutually absolutely continuous. 

\begin{definition}
A joint measure-preserving system ${\bm X} = (X, \mathcal{B}, \mu_{1}, \dots, \mu_{k}, T_{1}, \dots, T_{k})$ is \emph{pointwise jointly ergodic} if for any bounded measurable functions $f_{1}, \dots, f_{k}\in L^{\infty}(X)$,  
\[ 
\lim_{N \rightarrow \infty} \frac{1}{N} \sum_{n=1}^{N} \prod_{i=1}^{k} f_{i} (T_{i}^{n} x) = \prod_{i=1}^{k} \int_{X} f_{i} \, d \mu_{i}
\]
for almost every $x\in X$.
\end{definition} 

\begin{remark}\label{rem:pje_notation}~
\begin{enumerate}[(1)]
\item Since $\mu_{1}, \mu_{2}, \dots, \mu_{k}$ are equivalent, one has $L^{\infty} (X,\mu_{1}) = L^{\infty} (X,\mu_{2}) = \cdots = L^{\infty} (X,\mu_{k})$. 
Hence we just write $f_{1}, f_{2},\dots, f_{k}\in L^{\infty}(X)$ and say ``bounded measurable functions on $X$". 
In addition, we simply write ``almost every $x\in X$" instead of ``$\mu_{i}$-almost every $x\in X$ for every $i = 1, 2, \dots, k$". 
\item In our definition of joint ergodicity, the transformations are not required to be commuting or invertible and can have different invariant measures. However, if they do commute, then $\mu_{1} = \mu_{2} = \dots = \mu_{k}$; see Theorem \ref{th:coincidence} below. 
\item If $\mu_{1} = \mu_{2} = \dots = \mu_{k}$, then, letting $\mu$ be the common probability measure, we say that $T_{1}, T_{2},\dots, T_{k}$ are pointwise jointly ergodic on $(X, \mathcal{B}, \mu)$.
\end{enumerate}
\end{remark}
%We see that if $ {\bm X} = (X, \mathcal{B}, \mu_{1}, \dots, \mu_{k}, T_{1}, \dots, T_{k})$ is pointwisely joint ergodic, then it is $L^{2}$-jointly ergodic. 
In Section \ref{subsec:PJE}, we provide examples of pointwise jointly ergodic systems.

We now state our main result of this paper which extends the classical Wiener-Wintner theorem to pointwise jointly ergodic systems with nilsequences weight. 
%{\color{blue} See Section \ref{subsec:nilsequence} for the definition of nilsequences.}
%{\color{cyan} (9/16) nilsequences were already mentioned on line -4 of p.1.}

\begin{theorem}\label{th:mlww_pje_nil}
Suppose that a joint measure-preserving system ${\bm X}= (X, \mathcal{B}, \mu_{1}, \dots, \mu_{k}, T_{1}, \dots, T_{k})$ is pointwise jointly ergodic. 
Then, for any bounded measurable functions $f_{1}, \dots, f_{k}$ on $X$, there exists a full measure set $X_{f_{1}, \dots, f_{k}} \subset X$ such that the averages
\[ 
\frac{1}{N} \sum_{n=1}^{N} b_{n}\cdot f_{1} (T_{1}^{n} x) \cdots f_{k}(T_{k}^{n} x) 
\]
converge as $N\to \infty$ for every $x \in X_{f_{1}, \dots, f_{k}}$ and every nilsequence $(b_{n})_{n \in \mathbb{N}}$.
\end{theorem}

Now, we discuss several applications of Theorem \ref{th:mlww_pje_nil}. 
Firstly, this result allows us to obtain an application to the mean convergence of weighted ergodic averages.  
A bounded sequence $(w_{n})_{n \in \mathbb{N}}$ in $\mathbb{C}$ is called a \emph{universally good weight for the mean ergodic theorem} if for any system $(Y, \mathcal{D}, \nu,S)$ and any $\varphi \in L^{2}(Y,\nu)$, the limit 
\begin{equation}\label{def:good_weight}
\lim_{N\to \infty}\frac{1}{N} \sum_{n=1}^{N} w_{n}\cdot \varphi \circ S^{n}
\end{equation}
exists in $L^{2}(Y,\nu)$. 
As a consequence of the Wiener-Winter ergodic theorem \eqref{classical_ww_ergodic_average}, one can deduce that for every measure-preserving system $(X,\mathcal{B},\mu,T)$ and $f\in L^{1}(X,\mu)$, the sequence $(f(T^{n}x))_{n \in \mathbb{N}}$, with $x\in X_{f}$, is universally good weight for the mean ergodic theorem. 

Host and Kra \cite{Host-Kra2009} and subsequently Chu \cite{Chu2009} have generalized this result for multiple ergodic averages. 
By an \emph{integer polynomial} we mean a polynomial taking integer values on the integers. 
A bounded sequence $(w_{n})_{n \in \mathbb{N}}$ in $\mathbb{C}$ is called a \emph{universally good weight for the multiple polynomial ergodic theorem} if for any measure-preserving system $(Y, \mathcal{D}, \nu,S)$, $\varphi_{1},\dots \varphi_{r} \in L^{\infty}(Y,\nu)$ and all integer polynomials $p_{1},\dots ,p_{r}$, the limit 
\begin{equation}\label{def:good_weight_multi_poly}
\lim_{N\to \infty}\frac{1}{N} \sum _{n=1}^{N} w_{n}\cdot S^{p_{1}(n)}\varphi_{1}\cdots S^{p_{r}(n)}\varphi_{r}
\end{equation}
exists in $L^{2}(Y,\nu)$. 
For the case of $p_{\ell}(n)=\ell n$, it is called universally good weight for the multiple ergodic theorem. 
(Hence \eqref{def:good_weight} corresponds to the case when $r=1$ and $p_{1}(n)=n$ in \eqref{def:good_weight_multi_poly}.)
In \cite{Host-Kra2009}, Host and Kra showed that for every measure-preserving system $(X,\mathcal{B},\mu,T)$, $f\in L^{1}(X,\mu)$, and for almost every $x\in X$, the sequence $(f(T^{n}x))_{n \in \mathbb{N}}$ is universally good weight for the multiple ergodic theorem as a corollary of their generalized Wiener-Winter ergodic theorem. 
Subsequently, Chu \cite{Chu2009} proved that for almost every $x\in X$ the sequence $(f(T^{n}x))_{n \in \mathbb{N}}$ is universally good weight for the multiple polynomial ergodic theorem. 

In \cite{Chu2009}, Chu has given a convergence criterion for a bounded sequence to be a good weight for the multiple polynomial ergodic theorem. 

\begin{proposition*} \cite{Chu2009}*{Theorem 1.3} \label{prop:chu_criterion}
Given $r, d\in \mathbb{N}$, there exists an integer $m\in \mathbb{N}$ such that the following property holds: 
for a bounded sequence $(w_{n})_{n}$ in $\mathbb{C}$, if the limit 
\[ 
\lim_{N\to \infty}\frac{1}{N} \sum_{n=1}^{N} b_{n} w_{n}
\] 
exists for every $m$-step nilsequence $(b_{n})_{n}$, then for every measure-preserving system $(Y,\mathcal{D},\nu, S)$, $\varphi_{1},\dots \varphi_{r} \in L^{\infty}(Y,\nu)$, and all integer polynomials $p_{1},\dots ,p_{r}$ of degree at most $d$, the limit 
\[
\lim_{N\to \infty}\frac{1}{N} \sum _{n=1}^{N} w_{n}\cdot S^{p_{1}(n)}\varphi_{1}\cdots S^{p_{r}(n)}\varphi_{r}
\]
exists in $L^{2}(Y,\nu)$. 
\end{proposition*}

The following result is an immediate consequence of the proposition above and Theorem \ref{th:mlww_pje_nil}, which is an extension of Chu's result \cite{Chu2009}*{Theorem 1.1} for pointwise jointly ergodic systems.

\begin{corollary} \label{cor:goodweight_multi_poly}
Suppose that ${\bm X} = (X, \mathcal{B}, \mu_{1}, \dots, \mu_{k}, T_{1}, \dots, T_{k})$ is a pointwise jointly ergodic system. 
Let $f_{1},\dots ,f_{k}\in L^{\infty}(X)$. 
Then there exists a full measure set $X_{0}\subset X$ such that for every $x\in X_{0}$, the sequence $\{f_{1} (T_{1}^{n} x) \cdots f_{k}(T_{k}^{n} x)\}_{n}$ is a universally good weight for the multiple polynomial ergodic theorem: for every measure-preserving system $(Y, \mathcal{D}, \nu ,S)$, $\varphi_{1},\dots \varphi_{r} \in L^{\infty}(Y,\nu)$, and all integer polynomials $p_{1},\dots ,p_{r}$, the limit 
\[
\lim_{N\to \infty}\frac{1}{N} \sum _{n=1}^{N} f_{1} (T_{1}^{n} x) \cdots f_{k}(T_{k}^{n} x)\cdot S^{p_{1}(n)}\varphi_{1}\cdots S^{p_{r}(n)}\varphi_{r}
\]
exists in $L^{2}(Y,\nu)$. 
\end{corollary}

Secondly, Theorem \ref{th:mlww_pje_nil} also allows us to investigate almost everywhere convergence of the averages of the form
\begin{equation} \label{eq:average_alphan}
 \frac{1}{N} \sum_{n=1}^{N} \prod_{i=1}^{k} f_{i} (T_{i}^{\lfloor{ \alpha n\rfloor}} x)
\end{equation}
when ${\bm X} = (X, \mathcal{B}, \mu_{1}, \dots, \mu_{k}, T_{1}, \dots, T_{k})$ is pointwise jointly ergodic. 
Here, for $r\in \mathbb{R}$, let $\lfloor r\rfloor$ denote the greatest integer less than or equal to $r$. 
Note that although weights do not explicitly appear in the averages \eqref{eq:average_alphan}, these can be rewritten as $ \frac{1}{N} \sum\limits_{n=1}^{ \lfloor{ \alpha N \rfloor} } \mathbbm{1}_{A} (n) \prod\limits_{i=1}^{k} f_{i} (T_{i}^{n}x)$, where $\mathbbm{1}_{A}$ is the indicator function of the set $A = \{ \lfloor \alpha n\rfloor \colon n \in \mathbb{Z} \}$. 
Since $ \mathbbm{1}_{A} (n)$ can be approximated by trigonometric polynomials (see Lemma \ref{lem:Bap_modify} below), we apply (a specific case of) Theorem \ref{th:mlww_pje_nil} to derive Theorems \ref{thm:multiple_bap}, \ref{thm:equiv_multiple_bap} and Corollary \ref{cor:commutingPJE_TPJE} below.

\begin{theorem}
\label{thm:multiple_bap}
Let $\alpha \geq 1$ be a real number.
Suppose that ${\bm X} = (X, \mathcal{B}, \mu_{1}, \dots, \mu_{k}, T_{1}, \dots, T_{k})$ is pointwise jointly ergodic. 
Then for any bounded measurable functions $f_{1}, \dots, f_{k}$ on $X$, the averages
\begin{equation}
\label{eq:subsequence}
 \frac{1}{N} \sum_{n=1}^{N} \prod_{i=1}^{k} f_{i} (T_{i}^{\lfloor{ \alpha n\rfloor}} x)
\end{equation}
converge almost everywhere.
\end{theorem}

Moreover, we characterize the condition of $\alpha\geq 1$ when the limit of \eqref{eq:subsequence} is the ``natural" one, namely the product of the averages $ \prod\limits_{i=1}^{k} \int_{X} f_{i} \, d \mu_{i}$. 
Before formulating this result, let us introduce the following notation. 
Let $E(T)$ denote the set of eigenvalues of a measure-preserving transformation $T$. 
For an ergodic measure-preserving transformation $T$, it is known that $E(T)$ forms a countable multiplicative subgroup of $\mathbb{T} \,(:= \{ z \in \mathbb{C} \colon |z| =1  \})$. 
For a joint measure-preserving system $ {\bm X} = (X, \mathcal{B}, \mu_{1}, \dots, \mu_{k}, T_{1}, \dots, T_{k})$, we define $E(T_{1}, \dots, T_{k})$ as 
\[
E(T_{1}, \dots, T_{k}) = \left\{ \lambda \in \mathbb{T} \colon \lambda = \prod_{i=1}^{k} \lambda_{i}, \ \lambda_{i} \in E(T_{i})\right\}.
\]
It is known that $E(T_{1}, \dots, T_{k}) = E (T_{1} \times \cdots \times T_{k})$; see \cite{Furstenberg1981}*{Lemma 4.18}.

\begin{theorem}
\label{thm:equiv_multiple_bap}
Suppose that ${\bm X} = (X, \mathcal{B}, \mu_{1}, \dots, \mu_{k}, T_{1}, \dots, T_{k})$ is pointwise jointly ergodic.
Let $\alpha \geq 1$. 
Then the following are equivalent.
\begin{enumerate}[{\rm (1)}]
\item $E (T_{1}, \dots, T_{k}) \bigcap \{ \mathrm{e} (m/\alpha)\colon m \in \mathbb{Z} \}= \{1\}$. 
\item For any bounded measurable functions $f_{1}, \dots, f_{k}$ on $X$, 
\[ 
\lim_{N \rightarrow \infty} \frac{1}{N} \sum_{n=1}^{N}  f_{1}(T_{1}^{\lfloor \alpha n\rfloor} x) \cdots f_{k} (T_{k}^{\lfloor \alpha n\rfloor} x) = \prod_{i=1}^{k} \int_{X} f_{i} \, d\mu_{i}
\]
for almost every $x\in X$.
\end{enumerate}
\end{theorem}

Finally, we consider a pointwise jointly ergodic system  ${\bm X} = (X, \mathcal{B}, \mu_{1}, \dots, \mu_{k}, T_{1}, \dots, T_{k})$  when $T_{1}, \dots, T_{k}$ are commuting. 
In this case, as was mentioned in Remark \ref{rem:pje_notation}-(2), we have $\mu_{1} = \cdots = \mu_{k}$, and hence we let $\mu$ denote the common invariant measure. 

A measure-preserving transformation is called \emph{totally ergodic} if $T^{q}$ is ergodic for all $q \in \mathbb{N}$. 
Note that if $T$ is ergodic, then $T$ is totally ergodic if and only if $E (T) \cap \{ \mathrm{e} (a)\colon a \in \mathbb{Q} \} = \{1\}$. 
In a slight abuse of language, an eigenvalue $\lambda$ of the form $\lambda =\mathrm{e}(a)$ with $a\in \mathbb{Q}$ is called \emph{rational eigenvalue}. 
Hence for an ergodic transformation, it is totally ergodic if and only if it admits no rational eigenvalues other than $1$. 

In \cite{Berend1985}, Berend established that if commuting measure-preserving transformations $T_{1}, T_{2},\dots, T_{k}$, with $k\geq 2$, on a probability space $(X, \mathcal{B}, \mu)$ are $L^{2}$-jointly ergodic, then each $T_{i}$ is totally ergodic and, moreover, the transformations $T_{1}, T_{2},\dots, T_{k}$ are \emph{totally $L^{2}$-jointly ergodic}, meaning that $T_{1}^{q}, T_{2}^{q}, \dots, T_{k}^{q}$ are $L^{2}$-jointly ergodic for any $q\in \mathbb{N}$. 
In fact, it was shown that $T_{1} \times T_{2} \times \cdots \times T_{k}$ are totally ergodic with respect to $\mu^{\otimes k}$, hence we have that 
\[ 
E (T_{1}, T_{2}, \cdots T_{k} ) \bigcap \{ \mathrm{e}(a)\colon a \in \mathbb{Q} \} = \{1\}. 
\]
Thus, Theorem \ref{thm:equiv_multiple_bap} implies the following pointwise version of Berend's result.  

\begin{corollary} \label{cor:commutingPJE_TPJE}
Let $k \geq 2$. 
If $T_{1}, T_{2}, \dots, T_{k}$ are commuting and pointwise jointly ergodic on  $(X, \mathcal{B}, \mu)$, then they are totally pointwise jointly ergodic, that is, for any $q \in \mathbb{N}$, the joint measure-preserving system $(X,\mathcal{B},\mu_{1},\mu_{2},\dots ,\mu_{k},T_{1}^{q},T_{2}^{q}, \dots, T_{k}^{q})$ is pointwise jointly ergodic.  
\end{corollary}

The structure of this paper is as follows. 
In Section \ref{sec:pre} we provide some background materials, which will be used later.  
In Section \ref{sec3:pf1} we prove Theorem \ref{th:mlww_pje_nil}. 
Finally, Section  \ref{sec4:pf2} is devoted to the proofs of Theorems \ref{thm:multiple_bap} and \ref{thm:equiv_multiple_bap}.

\section{Preliminaries}
\label{sec:pre}

In this section, we introduce some basic facts and useful results, which we will use throughout the paper.  

\subsection{Van der Corput lemma}

One of the main tools in this paper is the following van der Corput trick; see, for example, \cite{Kuipers-Niederreiter1974}*{LEMMA 3.1, Chapter 1}. 

\begin{lemma}\label{lem:vdC}
Let $x_{1}, \dots, x_{N}$ be complex numbers, and let $H$ be an integer with $1 \leq H \leq N$. 
Then
\begin{equation}
\label{vdC-inequality}
H^{2} \left\vert \sum_{n=1}^{N} x_{n} \right\vert^{2} \leq
H (N+H-1) \sum_{n=1}^{N} |x_{n}|^{2} + 2 (N + H - 1) \sum_{h=1}^{H-1} (H-h) \mathrm{Re} \sum_{n=1}^{N-h} \langle x_{n+h}, x_{n} \rangle _{\mathbb{C}}.
\end{equation}
\end{lemma}
This lemma implies that if $(x_{n})_{n \in \mathbb{N}}$ is a bounded sequence in $\mathbb{C}$, then 
\[
\limsup_{N \rightarrow \infty} \left\vert \frac{1}{N} \sum_{n=1}^{N} x_{n} \right\vert^{2} \leq  \limsup_{H \rightarrow \infty} \frac{1}{H} \sum_{h=1}^{H} \limsup_{N \rightarrow \infty}  \frac{1}{N} \sum_{n=1}^{N} \mathrm{Re} \langle x_{n+h}, x_{n} \rangle _{\mathbb{C}}.
\]
In particular, if $ \gamma_{h} := \lim\limits_{N \rightarrow \infty} \frac{1}{N} \sum\limits_{n=1}^{N}  \langle x_{n+h}, x_{n} \rangle _{\mathbb{C}}$ exists for every $h\in \mathbb{N}$ and $\lim\limits_{H \rightarrow \infty} \frac{1}{H} \sum\limits_{h=1}^{H}  \gamma_{h} = 0$, then 
\[
\lim_{N \rightarrow \infty} \frac{1}{N} \sum\limits_{n=1}^{N} x_{n} =0.
\]

\subsection{Nilsystems and nilsequences} \label{subsec:nilsequence}

Let $m \in \mathbb{N}$. 
Assume that $G$ is an $m$-step nilpotent Lie group and $\Gamma <G$ is a discrete cocompact subgroup of $G$. 
The compact manifold $M=G/\Gamma$ is called an $m$-step nilmanifold. 
The group $G$ acts on $M$ by left translations $G\times M\ni (g,x)\mapsto gx\in M$. 
The Haar measure of $M$ is the unique probability measure invariant under this action, and is denoted by $\mu$.  
Letting $R$ denote left multiplication by the fixed element $\tau \in G$, we call $(M,\mathcal{B},\mu,R)$ an \emph{$m$-step nilsystem}, where $\mathcal{B}$ is the Borel $\sigma$-algebra of $M$. 

%We assume that $G$ is endowed with a right-invariant Riemannian metric $d_{G}$ such that the closed bounded sets are compact, and the nilmanifold $M$ is endowed with the quotient metric $d_{M}$. 
%Thus one can consider the pair $(M,R)$ as a topological dynamical system as well, and sometimes it is called \emph{topological $m$-step nilsystem} when we need to specify the underlying structure. 
%It is known that topological nilsystem is \emph{distal}, that is every distinct pair of points $x,y\in M$ satisfies $\inf_{n\in \mathbb{Z}}d_{M}(R^{n}(x),R^{n}(y))>0$; see \cite{Auslander-Green-Hahn1963}*{Chapter IV, \S 7, Theorem 3} and \cite{Leibman2005}*{\S 2.14}. 

Let $M=G/\Gamma$ be an $m$-step nilmanifold, $\varphi \colon M\to \mathbb{C}$ be a continuous function, $\tau \in G$, and $x\in M$. 
The sequence $(\varphi (\tau^{n}x))_{n\in \mathbb{N}}$ is called a \emph{basic $m$-step nilsequence}. 
Note also that the space of $m$-step nilsequences is an algebra under pointwise addition and multiplication. 
It is known that, given a polynomial $P$ of degree $d$, the sequence $(\mathrm{e}(P(n)))_{n\in \mathbb{N}}$ is a $d$-step nilsequence. 

\subsection{Host-Kra structure theory} \label{subsec:HKtheory}

We review the Host-Kra seminorms and an associated decomposition result of functions that we will use to prove Theorem \ref{th:mlww_pje_nil}. 
See \cites{Host-Kra2005, Host-Kra2009, Host-Kra2018} in detail for following discussion.

\subsubsection*{Notation and conventions}

For $z\in \mathbb{C}$, let $\mathrm{C}(z)=\bar{z}$ denote complex conjugation. 
Thus $\mathrm{C}^{m}z=z$ if $m$ is an even integer and $\mathrm{C}^{m}z=\bar{z}$ if $m$ is an odd integer. 
For $\epsilon =(\epsilon_{1}, \dots, \epsilon_{m})\in \{0, 1\}^{m}$, we write $\vert \epsilon \vert =\epsilon_{1}+\dots +\epsilon_{m}$. 
For $\epsilon =(\epsilon_{1}, \dots, \epsilon_{m})\in \{0, 1\}^{m}$ and $h=(h_{1}, \dots, h_{m})\in \mathbb{Z}^{m}$, we define $\epsilon \cdot h=\epsilon_{1}h_{1}+\dots +\epsilon_{m}h_{m}$.

For a measure-preserving system $(X,\mathcal{B} ,\mu , T)$, a \emph{factor} is a system $(Z,\mathcal{Z} ,\nu, S)$ along with a measurable map $\pi \colon X\to Z$ such that $\pi \circ T=S\circ \pi$ and $\pi_{*}\mu =\nu$.  
As usual, we use the same letter $T$ to denote the transformation in the factor system instead of $S$. 
In a slight abuse of language, we say $Z$ is a factor of $X$. 
Also, we use the same letter $\mathcal{Z}$ for both the $\sigma$-algebra of $Z$ and its inverse image $\pi^{-1}\mathcal{Z}\subset \mathcal{B}$. 
Hence, by an abuse of terminology, a $T$-invariant sub $\sigma$-algebra of $\mathcal{B}$ is also called factor.
We call a factor $Z$ of $X$ that admits the structure of nilsystem as a nilfactor.

Let $\mathcal{Z}(\subset \mathcal{B})$ be a factor. 
For $f\in L^{1}(X,\mu)$, let $\mathrm{E} (\,f\,|\,\mathcal{Z}\,)$ denote the conditional expectation of $f$ with respect to $\mathcal{Z}$. 

\subsubsection{Host-Kra seminorms} 

Let $(X,\mathcal{B} ,\mu , T)$ be an ergodic system and $f\in L^{\infty}(X,\mu)$.
Then for every $m\in \mathbb{N}$, define inductively
\begin{equation}\label{def:GHKseminorm}
\begin{split}
\VERT f \VERT_{1}^{2} 
&= \left| \int_X f \, d \mu \right|^2   \\
\VERT f \VERT_{m+1}^{2^{m+1}} &= \lim_{H \rightarrow \infty} \frac{1}{H} \sum_{h=0}^{H-1} \VERT T^{h} f \cdot \overline{f} \VERT_{m}^{2^m}. 
\end{split}
\end{equation}
%Note that $\VERT f \VERT_{1}=\left| \int_{X}f\, d\mu \right|$. 
It is shown in \cite{Host-Kra2005} that for every $m\in \mathbb{N}$ the limit above exists and $\VERT \cdot \VERT_{m}$ defines a seminorm on $L^{\infty}(X,\mu)$. 
We also have an explicit expression that 
\begin{equation}\label{eqn:HKseminorm2}
\VERT f \VERT_{m}^{2^{m}} 
=\lim_{H \rightarrow \infty} \frac{1}{H^{m}} \sum_{h_{1}, \dots, h_{m} =0}^{H-1} \int_{X} \prod_{\epsilon \in \{0, 1\}^{m} } \mathrm{C}^{\vert \epsilon \vert}  T^{\epsilon \cdot h} f \, d \mu.
\end{equation}
We sometimes write $\VERT \cdot \VERT_{m,T}$ instead of $\VERT \cdot \VERT_{m}$ when we need to specify the transformation for which the seminorms are considered. 

An ergodic system $(X,\mathcal{B}, \mu , T)$ is called weakly mixing if the product system $(X\times X,\mathcal{B} \times \mathcal{B} ,\mu \otimes \mu , T\times T)$ is ergodic. 
If the system $(X,\mathcal{B} ,\mu , T)$ is weakly mixing, then for every $m\in \mathbb{N}$ and $f\in L^{\infty}(X,\mu)$, we have $\VERT f \VERT_{m}=\vert \int_{X} f \, d \mu \vert$. 

We conclude this subsection by providing the following lemma.
\begin{lemma}\label{lem:product}
Let $(X, \mathcal{B}, \mu, T)$ and $(Y, \mathcal{D}, \nu, S)$ are measure-preserving systems. 
Suppose that $T \times S$ is ergodic with respect to $\mu \otimes \nu$. 
Then for every $m\in \mathbb{N}$, $f\in L^{\infty}(X,\mu)$ and $g\in L^{\infty}(Y,\nu)$, we have
\[ 
\VERT f \otimes g \VERT_{m, T \times S} \leq \VERT f \VERT_{m, T} \| g \|_{L^{\infty}(Y,\nu)}. 
\]
\end{lemma}
\begin{proof}
The proof proceeds by induction $m\in \mathbb{N}$. 
Since $T \times S$ is ergodic with respect to $\mu \otimes \nu$, one has 
\begin{equation*}
\VERT f \otimes g \VERT_{1, T \times S}=\left\vert \int_{X\times Y} f\otimes g \, d \mu \otimes \nu \right\vert =\left\vert \int_{X} f\, d \mu \right\vert \cdot \left\vert \int_{Y} g \, d\nu \right\vert \leq \VERT f \VERT_{1, T} \cdot \left\| g\right\|_{L^{\infty}(Y,\nu)}. 
\end{equation*}

Let $m\geq 2$ and assume that the result holds for $m-1$. 
Then one has 
\[
\VERT (T\times S)^{h}(f \otimes g)\cdot \overline{f \otimes g}\VERT_{m-1,T\times S}\leq \VERT T^{h}f \cdot \overline{f}\VERT_{m-1,T} \cdot \| S^{h}g\cdot \overline{g}\|_{L^{\infty}(Y,\nu)} \leq  \VERT T^{h}f \cdot \overline{f}\VERT_{m-1,T} \cdot  \| g\|_{L^{\infty}(Y,\nu)}^{2}, 
\]
and hence 
\begin{align*}
\frac{1}{H} \sum_{h=0}^{H-1} \VERT (T\times S)^{h} f \otimes g \cdot \overline{f \otimes g} \VERT_{m-1}^{2^{m-1}}\leq \frac{1}{H} \sum_{h=0}^{H-1} \VERT T^{h}f \cdot \overline{f}\VERT_{m-1,T}^{2^{m-1}} \cdot \| g\|_{L^{\infty}(Y,\nu)}^{2^{m}}. 
\end{align*}
It follows that 
\begin{align*}
\VERT f \otimes g \VERT_{m, T \times S}^{2^{m}}
&=\lim_{H \to \infty} \frac{1}{H} \sum_{h=0}^{H-1} \VERT (T\times S)^{h} f \otimes g \cdot \overline{f \otimes g} \VERT_{m-1}^{2^{m-1}} \\ 
&\leq \left[ \lim_{H \to \infty} \frac{1}{H} \sum_{h=0}^{H-1} \VERT T^{h}f \cdot \overline{f}\VERT_{m-1,T}^{2^{m-1}} \right] \| g\|_{L^{\infty}(Y,\nu)}^{2^{m}} =\VERT f \VERT_{m, T}^{2^{m}}\cdot \| g\|_{L^{\infty}(Y,\nu)}^{2^{m}} 
\end{align*}
The proof of Lemma \ref{lem:product} is obtained. 
\end{proof}

\subsubsection{The ergodic structure theorem}

Let $m\in \mathbb{N}$. 
For an invertible ergodic system $(X, \mathcal{B}, \mu, T)$, there exists a factor $\mathcal{Z}_{m} \subset \mathcal{B}$ of order $m$ of $X$ such that 
 \begin{enumerate}[(i)]
 \item for $f \in L^{\infty}$, one has $\VERT f \VERT_{m+1} =0$ if and only if $\mathrm{E} (\,f\,|\,\mathcal{Z}_{m}\,) = 0$, so there exists the orthogonal decomposition
\begin{equation}
\label{HK-decomposition} L^2(X) = L^{2} (Z_{m}) \oplus \overline{ \{f \in L^{\infty} (X)\colon \VERT f \VERT_{m+1} =0  \}}
\end{equation}
 \item the system $Z_{m} = Z_{m}(X) = (X, \mathcal{Z}_{m}, \mu, T)$ is an inverse limit of $m$-step ergodic nilsystems. 
\end{enumerate}

This leads to the following useful structure theorem. 
See \cite{Host-Kra2018}*{Theorem 5, Section 16.1}. 
\begin{theorem}
\label{thm:HK-structure}
Let $(X, \mathcal{B}, \mu, T)$ be an ergodic system. 
Let $m \in \mathbb{N}$, $\epsilon > 0$ and $f \in L^{1}(X,\mu)$. 
Then there exist an $m$-step nilsystem $(Z, \mathcal{Z}, \nu, T)$ and a factor map $\pi \colon X \rightarrow Z$ such that 
\[ 
f = f_{\rm nil} \circ \pi + f_{\rm unif} + f_{\rm sml},
\]
where 
\begin{enumerate}[{\rm (i)}]
\item the function $f_{\rm nil}$ is a continuous function on $Z$, and if we further assume $f\in L^{\infty}(X,\mu)$, then $\| f_{\rm nil} \|_{\infty}\leq \| f_{\rm sml} \|_{L^{\infty}(X,\mu)}$;
\item  the function $f_{\rm unif} \in L^{\infty}(X,\mu)$ satisfies $\VERT f_{\rm unif} \VERT_{m+1} = 0$;
\item  the function $f_{\rm sml} \in L^{1}(X,\mu)$ satisfies $\| f_{\rm sml} \|_{L^{1}(X,\mu)} < \epsilon$.
\end{enumerate}
\end{theorem} 

\begin{remark}
Although the above structure theorem in \cite{Host-Kra2018} assumes that the system $ (X, \mathcal{B}, \mu, T)$ is invertible, it still holds for non-invertible systems.  
To see this, let us consider the natural extension $(\widetilde{X}, \widetilde{\mathcal{B}}, \widetilde{\mu}, \widetilde{T})$ of the system $(X, \mathcal{B}, \mu, T)$. 
By \eqref{HK-decomposition}, one has
\begin{equation*}
L^{2}(\widetilde{X}) = L^{2} (Z_{m}(\widetilde{X})) \oplus \overline{ \left\{f \in L^{\infty} (\widetilde{X})\colon \VERT f \VERT_{m+1, \widetilde{T}} =0  \right\}},
\end{equation*}
and hence
\begin{equation}
\label{HK-decomposition-noninvertible}
L^{2}(X) = \left[ L^{2} (Z_{m}(\widetilde{X})) \cap L^{2}(X)\right] \oplus \left[ \overline{ \left\{f \in L^{\infty} (\widetilde{X})\colon \VERT f \VERT_{m+1, \widetilde{T}} =0  \right\}} \cap L^{2}(X)\right].
\end{equation}
Note that \eqref{def:GHKseminorm} is also defined for $(\widetilde{X}, \widetilde{\mathcal{B}}, \widetilde{\mu}, \widetilde{T})$ and it is not hard to check that 
\begin{equation*}
\overline{ \left\{f \in L^{\infty} (\widetilde{X})\colon \VERT f \VERT_{m+1, \tilde{T}} =0  \right\}} \cap L^{2}(X) =  \overline{ \left\{ f \in L^{\infty} (X)\colon \VERT f \VERT_{m+1, T} =0  \right\}}.
\end{equation*}
Recall that any factor of a nilfactor is also a nilfactor. 
Therefore, Theorem \ref{thm:HK-structure} for noninvertible systems follows from \eqref{HK-decomposition-noninvertible}.
\end{remark}

\subsubsection{Local seminorms on $\ell^{\infty}$} 

In this subsection, we review the local seminorms introduced by Host and Kra in \cite{Host-Kra2009}. 
Let $m \in \mathbb{N}$. 
Let $a=(a_{n})_{n\in \mathbb{N}}$ be a bounded sequence in $\mathbb{C}$, and ${\bm I}=(I_{N})_{N\in \mathbb{N}}$ be a sequence of intervals in $\mathbb{Z}$ whose lengths tend to infinity as $N\to \infty$. 
We say that $(a_{n})_{n\in \mathbb{N}}$ satisfies property $\mathcal{P}(m)$ on ${\bm I}$ if for any $h=(h_{1}, \dots, h_{m}) \in ({\mathbb{N} \cup \{0\}})^{m}$, the following limit 
\[ 
\lim_{N\rightarrow \infty} \frac{1}{\vert I_{N}\vert} \sum_{n\in I_{N}} \prod_{\epsilon \in \{0,1\}^{m}} \mathrm{C}^{\vert \epsilon \vert} a_{n + h \cdot \epsilon}
\]
exists, where $\vert I_{N}\vert$ denotes the length of the interval $I_{N}$. 
Let $c_{h} (a,{\bm I})$ denote this limit. 
It is shown in \cite{Host-Kra2009}*{Proposition 2.2} that the following limit exists and is non-negative: 
\[
\lim_{H \rightarrow \infty} \frac{1}{H^{m}} \sum_{h_{1}, \dots, h_{m} = 0}^{H-1} c_{h} (a,{\bm I}). 
\] 
Then one defines 
\begin{equation}\label{def:loc_seminorm}
\| a \|_{m,{\bm I}} = \left[ \lim_{H \rightarrow \infty} \frac{1}{H^{m}} \sum_{h_{1}, \dots, h_{m} = 0}^{H-1} c_{h} (a,{\bm I}) \right]^{1/2^{m}},
\end{equation}
which is called a local seminorm in \cite{Host-Kra2009}*{Definition 2.3}. 
For the case $I_{N}=\{ 1,\dots ,N\}$, the property is simply referred to as property $\mathcal{P}(m)$ with omitting `on ${\bm I}$' from the nomenclature, and accordingly we simply write $c_{h} (a)$ and $\| a \|_{m}$ instead of $c_{h} (a,{\bm I})$ and $\| a \|_{m,{\bm I}}$, respectively. 

The following result is obtained in \cite{Host-Kra2009}*{Corollary 2.14}. 
We write the uniform norm of a sequence $a= (a_{n})_{n}$ as $\| a \|_{\infty}$. 
\begin{lemma}
\label{lemma:HostKra:2.14}
Let $b = (b_{n})_{n}$ be an  $(m-1)$-step nilsequence and $\delta > 0$. There exists a constant $c = c(b, \delta)$ such that for every bounded sequence $a = (a_{n})_{n}$ satisfying property $\mathcal{P}(m)$, we have 
\[ \limsup_{N \rightarrow \infty} \left| \frac{1}{N} \sum_{n=1}^{N} a_{n} b_{n} \right| \leq c \| a \|_{m} + \delta \| a \|_{\infty}. \]
\end{lemma}

\subsection{Pointwise jointly ergodic systems}\label{subsec:PJE}
 
Recall that a joint measure-preserving system ${\bm X} = (X, \mathcal{B}, \mu_{1}, \dots, \mu_{k}, T_{1}, \dots, T_{k})$ is pointwise jointly ergodic if for any bounded measurable functions $f_{1}, \dots, f_{k}\in L^{\infty}(X)$,  
\[ 
\lim_{N \rightarrow \infty} \frac{1}{N} \sum_{n=1}^{N} \prod_{i=1}^{k} f_{i} (T_{i}^{n} x) = \prod_{i=1}^{k} \int_{X} f_{i} \, d \mu_{i}
\]
for almost every $x\in X$.
Note that if $ {\bm X}$ is pointwise jointly ergodic, then each $(X, \mathcal{B}, \mu_{i}, T_{i})$ is ergodic and, moreover, $T_{1} \times \cdots \times T_{k}$ is ergodic with respect to the product measure $\otimes_{i=1}^{k}\mu_{i}$; see \cite{Bergelson-Son2023}*{Theorem 1.8}. 

\begin{example}~
\begin{enumerate}[{\rm (1)}]
\item Disjoint systems: A \emph{joining} of measure-preserving systems $(X_{i}, \mathcal{B}_{i}, \mu_{i}, T_{i})$, $1 \leq i \leq k$, is a probability measure $\lambda$ on $X_{1} \times X_{2} \times \cdots \times X_{k}$ invariant under $T_{1} \times T_{2} \times \cdots \times T_{k}$ such that the projection of $\lambda$ onto $X_{i}$ is $\mu_{i}$ for $i= 1, 2,\dots, k$.  
We call measure-preserving systems $(X_{i}, \mathcal{B}_{i}, \mu_{i}, T_{i})$, $1 \leq i \leq k$, \emph{disjoint} if the product measure $\otimes_{i=1}^{k}\mu_{i}$ is the only joining of systems. 
See \cite{Furstenberg1967} and \cite{Rudolph1990} for details. 

Let ${\bm X} = (X, \mathcal{B}, \mu_{1}, \mu_{2}, \dots, \mu_{k}, T_{1}, T_{2}, \dots, T_{k})$ be a joint measure-preserving system, where $X$ is a compact metric space and $\mathcal{B}$ is the Borel $\sigma$-algebra of $X$. 
It is known that  if $(X, \mathcal{B}, \mu_{i}, T_{i}), 1 \leq i \leq k,$ are ergodic and disjoint, then ${\bm X}$ is pointwise jointly ergodic; see  \cite{Berend1985}*{Theorem 2.2} and \cite{Bergelson-Son2022}*{Corollary 5.3}. 
%{\color{blue} See also \cite{HKS2022} for several multiple ergodic properties of disjoint systems.} 

\item Exact endomorphism or $K$-automorphism: In \cite{Derrien-Lesigne1996},  Derrien and Lesigne demonstrated that if $T$ is an exact endomorphism or a $K$-automorphism on a probability space $(X, \mathcal{B}, \mu)$, then $T, T^{2}, \dots, T^{k}$ are pointwise jointly ergodic on $(X, \mathcal{B}, \mu)$.

\item Piecewise monotone maps on the unit interval: Let $X = [0, 1]$ and $\mathcal{B}$ be the Borel $\sigma$-algebra of $X$.  
In \cite{Bergelson-Son2022}, the phenomena of pointwise joint ergodicity of piecewise monotone maps was investigated. 
Among other things, it was shown that times $b$-map $T_{b} x = bx \pmod 1$ and Gauss map $T_{G}(x) = \frac{1}{x} \pmod 1$ are pointwise jointly ergodic. 
More precisely, the joint measure-preserving system ${\bm X} = (X, \mathcal{B}, \mu, \mu_{G}, T_{b}, T_{G})$ is pointwise jointly ergodic, where $\mu$ is the Lebesgue measure on $[0,1]$ and $\mu_{G}$ is the Gauss measure on $[0,1]$ given by $\mu_{G}(A) = \frac{1}{\log 2} \int_{A} \frac{1}{1+t} \, dt$, respectively.
\end{enumerate}
\end{example}

As in the examples above, the phenonmenon of joint ergodicity take place for the case where the involved transformations may not commute. 
However, the following result implies that if  $T_{1}, \dots, T_{k}$ are commuting, then they share the same invariant measure. 

\begin{theorem}\label{th:coincidence}
Let ${\bm X} = (X, \mathcal{B}, \mu_{1}, \dots, \mu_{k}, T_{1}, \dots, T_{k})$ be a pointwise jointly ergodic system. 
Suppose that $T_{1}, \dots, T_{k}$ are commuting. 
Then one has $\mu_{1}= \dots =\mu_{k}$. 
\end{theorem}

This result is an immediate consequence of the following proposition.

\begin{proposition}
\label{prop:commutingcase}
Let $(X, \mathcal{B}, \mu, T)$ and $(X, \mathcal{B}, \nu, S)$ be ergodic measure-preserving systems. 
If $T, S$ commute and $\mu, \nu$ are mutually absolutely continuous, then $\mu = \nu$. 
\end{proposition}

\begin{comment}
To prove the above proposition, we need the following useful fact. 

\begin{lemma}[cf. \cite{Postnikov-Pyateckii1957}*{Theorem 1}]
\label{equivalent_measure}
Let $(X, \mathcal{B}, \mu, T)$ be an ergodic measure-preserving system. 
Suppose that $\nu$ is a $T$-invariant probability measure on $(X, \mathcal{B})$ with $\nu \ll \mu$. 
Then $\nu = \mu$.
\end{lemma}
\end{comment}

\begin{proof}
First, we show $\nu$ is invariant under $T$. 
Note that for any measurable set $A\in \mathcal{B}$, we have
\[ 
\int_{X} \mathbbm{1}_{A} (S^{n} x) \, d \mu = \mu (S^{-n} A ) = \mu (T^{-1} S^{-n} A) = \mu (S^{-n} T^{-1}A ) = \int_{X} \mathbbm{1}_{T^{-1} A} (S^{n} x) \, d \mu,
\]
and thus 
\[ 
\int_{X} \frac{1}{N} \sum_{n=1}^{N} \mathbbm{1}_{A} (S^{n} x) \, d \mu= \int_{X} \frac{1}{N} \sum_{n=1}^{N} \mathbbm{1}_{T^{-1} A} (S^{n} x) \, d \mu. 
\]
Since $S$ is ergodic with respect to $\nu$, we have
\[
\int_{X} \frac{1}{N} \sum_{n=1}^{N} \mathbbm{1}_{A} (S^{n} x) \, d \mu \to \nu (A)\quad \text{and} \quad \int_{X} \frac{1}{N} \sum_{n=1}^{N} \mathbbm{1}_{T^{-1} A} (S^{n} x) \, d \mu \to \nu(T^{-1} A)
\]
as $N\to \infty$, respectively. 
It then follows that $\nu(A) = \nu(T^{-1} A)$, which means $\nu$ is $T$-invariant. 

%By Lemma \ref{equivalent_measure}, we have $\nu = \mu$.
Since $\mu$ is ergodic for $T$ and $\nu$ is absolutely continuous with respect to $\mu$, we have $\nu = \mu$.
\end{proof}

\section{Proof of Theorem \ref{th:mlww_pje_nil}}
\label{sec3:pf1}

The main goal of this section is to prove Theorem \ref{th:mlww_pje_nil}. 
First, we will prove the following proposition.

\begin{proposition} \label{prop:conv_unif}
Let ${\bm X}= (X, \mathcal{B}, \mu_{1}, \dots, \mu_{k}, T_{1}, \dots, T_{k})$ be a pointwise jointly ergodic system. 
Let $m\geq 2$ be an integer. 
If $f_{1}, \dots, f_{k}\in L^{\infty}(X)$ and $\VERT  f_{i} \VERT_{m, T_{i}} = 0$ for some $i\in \{1,\dots ,k\}$,
then there exists a full measure set $W_{m}\subset X$ such that if $ b = (b_{n})_{n}$ is an $(m-1)$-step nilsequence and $x \in W_{m}$, 
\[ \lim_{N \rightarrow \infty} \frac{1}{N} \sum_{n=1}^{N} b_{n} \prod_{i=1}^{k} f_{i} (T_{i}^{n} x) = 0. \]
\end{proposition}

\begin{proof}[Proof of Proposition \ref{prop:conv_unif}]
Without loss of generality, we assume that $\| f_{i} \|_{L^{\infty} (X) } \leq 1$ for all $i=1, 2, \dots, k$. 
Given $x\in X$ and $n\in \mathbb{N}$, we let $a_{n}(x)=\prod_{i=1}^{k} f_{i}(T_{i}^{n} x)$. 
Since ${\bm X}$ is pointwise jointly ergodic, for any $h=(h_{1}, \dots, h_{m}) \in (\mathbb{N} \cup \{0\})^{m}$, there exists a full measure set $W_{m}(h)\subset X$ such that 
\begin{equation}\label{eqn:c_h}
\begin{split}
c_{h}(a(x))=\lim_{N \rightarrow \infty} \frac{1}{N} \sum_{n=1}^{N}  \prod_{\epsilon \in \{0,1\}^{m}} \mathrm{C}^{\vert \epsilon \vert} a_{n+h\cdot \epsilon} (x) 
&=\lim_{N \rightarrow \infty} \frac{1}{N} \sum_{n=1}^{N} \prod_{i=1}^{k} \prod_{\epsilon \in \{0,1\}^{m}} \mathrm{C}^{\vert \epsilon \vert} f_{i} (T_{i}^{n + h \cdot \epsilon}x) \\
&=\prod_{i=1}^{k}   \int_{X}  \prod_{\epsilon \in \{0,1\}^{m}} \mathrm{C}^{\vert \epsilon \vert} f_{i} \circ T_{i}^{h \cdot \epsilon} \, d \mu_{i}
\end{split}
\end{equation}
for every $x\in W_{m}(h)$. 
This means that the sequence $a(x)=(a_{n}(x))_{n}$ satisfies property $\mathcal{P}(m)$. 

Let $W_{m}=\bigcap_{h}W_{m}(h)$. 
Then $W_{m}\subset X$ is a full measure set. 
By Lemma \ref{lemma:HostKra:2.14}, for any $\delta > 0$, there exists $c= c(b, \delta)$ such that
\[ 
\limsup_{N \rightarrow \infty} \left|  \frac{1}{N} \sum_{n=1}^{N} a_{n}(x)b_{n} \right| \leq c \| a(x) \|_{m} + \delta 
\]
for every $x\in W_{m}$. 
Here note that $\| a(x) \|_{m}=\VERT f_{1}\otimes \dots \otimes f_{k}\VERT_{m,T_{1}\times \dots \times T_{k}}$. 
Indeed, in view of \eqref{def:loc_seminorm}, \eqref{eqn:c_h} and \eqref{eqn:HKseminorm2}, we have
\begin{align*}
\| a(x) \|_{m}^{2^{m}} 
&=\lim_{H \rightarrow \infty} \frac{1}{H^{m}} \sum_{h_{1}, \dots, h_{m} = 0}^{H-1} c_{h} (a(x)) \\
&=\lim_{H \rightarrow \infty} \frac{1}{H^{m}} \sum_{h_{1}, \dots, h_{m} = 0}^{H-1} \prod_{i=1}^{k} \int_{X}  \prod_{\epsilon \in \{0,1\}^{m}} \mathrm{C}^{\vert \epsilon \vert} f_{i} \circ T_{i}^{h \cdot \epsilon} \, d \mu_{i} \\
&=\lim_{H \rightarrow \infty} \frac{1}{H^{m}} \sum_{h_{1}, \dots, h_{m} = 0}^{H-1} \int_{X^{k}}  \prod_{\epsilon \in \{0,1\}^{m}} \mathrm{C}^{\vert \epsilon \vert} f_{1}\otimes \dots \otimes f_{k} \circ (T_{1}\times \dots \times T_{k})^{h \cdot \epsilon} \, d\mu_{1}\otimes \dots \otimes d\mu_{k} \\
&=\VERT f_{1}\otimes \dots \otimes f_{k}\VERT_{m,T_{1}\times \dots \times T_{k}}^{2^{m}}.  
\end{align*}

Since it is assumed that $\VERT  f_{i} \VERT_{m, T_{i}} = 0$ for some $i\in \{1,\dots ,k\}$, it follows from Lemma \ref{lem:product} that $\| a(x) \|_{m}=0$ for every $x\in W_{m}$, and hence
\[ 
\limsup_{N \rightarrow \infty} \left|  \frac{1}{N} \sum_{n=1}^{N} a_{n}(x)b_{n} \right| \leq  \delta 
\]
Then the result follows since $\delta >0$ can be chosen arbitrarily small. 
\end{proof}

Now we prove Theorem \ref{th:mlww_pje_nil}.
\begin{proof}[Proof of Theorem \ref{th:mlww_pje_nil}]
 For a joint measure-preserving system ${\bm X}=(X,\mathcal{B}, \mu_{1}, \dots, \mu_{k}, T_{1}, \dots, T_{k})$, we say that $k$ is order of ${\bm X}$.
The proof proceeds by induction on $k \in \mathbb{N}$. 
%By a generalized Wiener-Wintner theorem \cite{Host-Kra2009}*{Theorem 2.22}, the case for $k=1$ holds.
The case $k=1$ is the Wiener-Wintner theorem for nilsequences due to Host and Kra \cite{Host-Kra2009}*{Theorem 2.22}: there exists a full measure set $X_{f_{1}}\subset X$ such that for every $x\in X_{f_{1}}$, the averages 
\[
\frac{1}{N} \sum_{n=1}^{N} b_{n} \cdot f_{1} (T_{1}^{n} x)
\]
converges as $N \to \infty$ for every nilsequence $(b_{n})_{n}$. 

Suppose that the statement holds up to any pointwise jointly ergodic measure-preserving system with order $(k-1)$ for some integer $k \geq 2$.
Now consider a pointwise jointly ergodic system ${\bm X}=(X,\mathcal{B}, \mu_{1}, \dots, \mu_{k}, T_{1}, \dots, T_{k})$ of order $k\geq 2$.  
Let $f_{1}, f_{2}, \dots, f_{k} \in L^{\infty}(X)$. 
Without loss of generality, we assume that $\| f_{i} \|_{L^{\infty}} \leq 1$ for $i = 1, 2, \dots, k$.
By the induction hypothesis, there exists a full measure set $X_{f_{2}, \dots, f_{k}} \subset X$ such that for every $x \in X_{f_{2}, \dots, f_{k}} $, the averages 
\[
\frac{1}{N} \sum_{n=1}^{N} b_{n} \cdot f_{2} (T_{2}^{n} x) \cdots f_{k} (T_{k}^{n} x)
\]
converges as $N \to \infty$ for every nilsequence $b=(b_{n})_{n}$. 

Fix an integer $m \in \mathbb{N}$. 
In view of Theorem \ref{thm:HK-structure} (the Host-Kra ergodic structure theorem) for the ergodic system $(X,\mu_{1},T_{1})$, for any $r \in \mathbb{N}$ one has
\[ 
f_{1} = u^{(r)} + v^{(r)}+ f_{\rm sml}^{(r)},
\]
where 
\begin{enumerate}[(i)]
\item $\| u^{(r)}\|_{L^{\infty}(X)}\leq \| f_{1}\|_{L^{\infty}(X)}$ and $(u^{(r)} (T_{1}^{n} x))_{n}$ is an $(m-1)$-step nilsequence for almost every $x\in X$;
\item the function $v^{(r)}\in L^{\infty}(X)$ satisfies $\VERT v^{(r)} \VERT_{m,T_{1}} = 0$;
\item the function $f_{\rm sml}^{(r)} \in L^{1}(X)$ satisfies $\| f_{\rm sml}^{(r)} \|_{L^{1}(X)} < 1/r^{3}$. 
\end{enumerate}

Let $Y_{1,m}\subset X$ be a full measure set such that if $x \in Y_{1,m}$, then $(u^{(r)} (T_{1}^{n} x))_{n}$ is an $(m-1)$-step nilsequence for any $r \in \mathbb{N}$. 

By Proposition \ref{prop:conv_unif}, there exists a full measure set $Y_{2,m}\subset X$ such that for $x \in Y_{2,m}$,
\begin{equation}\label{eq:unif_vanish}
\lim_{N\to \infty}\frac{1}{N} \sum_{n=1}^{N} b_{n} \cdot v^{(r)} (T_{1}^{n} x) f_{2} (T_{2}^{n} x)  \cdots f_{k}(T_{k}^{n} x) =0
\end{equation}
for all $(m-1)$-step nilsequence $(b_{n})_{n}$ and for any $r \in \mathbb{N}$. 

Define $B^{*} = \limsup\limits_{r\to \infty} B_{r} =\bigcap_{q=1}^{\infty} \bigcup_{r=q}^{\infty} B_{r}$, where 
\[
B_{r} = \left\{ x\in X\colon  \sup\limits_{N \in \mathbb{N}} \frac{1}{N} \sum\limits_{n=1}^{N} \left\vert f_{\rm sml}^{(r)} \right\vert (T_{1}^{n} x) > \frac{1}{r} \right\} 
\]
for $r\in \mathbb{N}$.
Let $Y_{3,m}=X\setminus B^{*}$. 
By the maximal ergodic inequality and item (iii) above,  
\[
\mu_{1} (B_r) \leq r \| f_{\rm sml}^{(r)} \|_{L^{1}(\mu_{1})} \leq \frac{1}{r^{2}},
\]
and hence $\sum\limits_{r=1}^{\infty}\mu_{1}(B_{r}) < \infty$. 
Then, by the Borel-Cantelli lemma, one has $\mu_{1}(B^{*})=0$. 
Thus $Y_{3,m}$ is a full measure set. 

Let $X_{m} = X_{f_{2}, \dots, f_{k}} \cap Y_{1,m} \cap Y_{2,m} \cap Y_{3,m}$. 
Let us show that for any $x \in X_{m}$, the averages 
\begin{equation}\label{eq:proof:average}
\frac{1}{N} \sum_{n=1}^{N} b_{n} \cdot f_{1}(T_{1}^{n} x) f_{2} (T_{2}^{n} x) \cdots f_{k} (T_{k}^{n} x)
\end{equation}
converge as $N \to \infty$ for any $(m-1)$-step nilsequence $(b_{n})_{n}$.
Define
\[
F_{M,N,b}(x) = \left\vert \frac{1}{N} \sum_{n=1}^{N} b_{n} \prod_{i=1}^{k}f_{i} (T_{i}^{n} x) - \frac{1}{M} \sum_{n=1}^{M} b_{n} \prod_{i=1}^{k}f_{i} (T_{i}^{n} x) \right\vert.
\]
By the triangle inequality, we have
\begin{align}
F_{M,N,b}(x)
&\leq \left\vert \frac{1}{N} \sum_{n=1}^{N} b_{n} \cdot u^{(r)}  (T_{1}^{n} x)\prod_{i=2}^{k}f_{i} (T_{i}^{n} x) - \frac{1}{M} \sum_{n=1}^{M} b_{n} \cdot u^{(r)}  (T_{1}^{n} x)\prod_{i=2}^{k}f_{i} (T_{i}^{n} x) \right\vert  \label{ineq:nil} \\
&\quad +\left\vert \frac{1}{N} \sum_{n=1}^{N} b_{n} \cdot v^{(r)} (T_{1}^{n} x)\prod_{i=2}^{k}f_{i} (T_{i}^{n} x)\right\vert + \left\vert\frac{1}{M} \sum_{n=1}^{M} b_{n} \cdot v^{(r)} (T_{1}^{n} x)\prod_{i=2}^{k}f_{i} (T_{i}^{n} x) \right\vert \label{ineq:unif}\\
&\quad +\left\vert \frac{1}{N} \sum_{n=1}^{N} b_{n} \cdot f_{\rm sml}^{(r)} (T_{1}^{n} x)\prod_{i=2}^{k}f_{i} (T_{i}^{n} x) - \frac{1}{M} \sum_{n=1}^{M} b_{n} \cdot f_{\rm sml}^{(r)} (T_{1}^{n} x)\prod_{i=2}^{k}f_{i} (T_{i}^{n} x) \right\vert . \label{ineq:sml}
\end{align}

Since $(u^{(r)} (T_{1}^{n} x))_{n}$ is an $(m-1)$-step nilsequence for every $x \in X_{m} \subset Y_{1, m}$, so is the sequence $(b_{n}\cdot u^{(r)} (T_{1}^{n} x))_{n}$. 
Then, for $x \in X_{m} \subset X_{f_{2}, \dots, f_{k}}$, the right-hand side of \eqref{ineq:nil} vanishes as $M, N \rightarrow \infty$. 
Also if $x \in X_{m} \subset Y_{2, m}$, then, by \eqref{eq:unif_vanish}, two terms in \eqref{ineq:unif} vanish as $M, N \rightarrow \infty$.
Therefore, for every $x \in X_{m}$ and $r \in \mathbb{N}$,
\begin{align*}
\limsup_{M, N \rightarrow \infty} F_{M,N,b}(x) 
&\leq \limsup_{M, N \rightarrow \infty} \left| \frac{1}{N} \sum_{n=1}^{N} b_{n} \cdot f_{\rm sml}^{(r)} (T_{1}^{n} x) \prod_{i=2}^{k}f_{i} (T_{i}^{n} x)  - \frac{1}{M} \sum_{n=1}^{M} b_{n} \cdot f_{\rm sml}^{(r)} (T_{1}^{n} x) \prod_{i=2}^{k}f_{i} (T_{i}^{n} x) \right| \\
&\leq 2 \sup_{N\in \mathbb{N}} \left| \frac{1}{N} \sum_{n=1}^{N} b_{n} \cdot f_{\rm sml}^{(r)} (T_{1}^{n} x) \prod_{i=2}^{k}f_{i} (T_{i}^{n} x) \right|
\leq 2 \| b\|_{\infty} \sup_{N\in \mathbb{N}} \frac{1}{N} \sum_{n=1}^{N} \left\vert f_{\rm sml}^{(r)} \right\vert (T_{1}^{n} x)
\end{align*}
since $\| f_{i} \|_{L^{\infty}} \leq 1$.
For $x \in X_{m} \subset Y_{3, m}$, there exists $q \in \mathbb{N}$ such that $x \notin B_{r}$ for any $r \geq q$, thus 
\begin{equation}
\limsup_{M, N \rightarrow \infty} F_{M,N,b}(x)  \leq 2 \| b\|_{\infty} \sup_{N \in \mathbb{N}} \frac{1}{N} \sum_{n=1}^{N} \left\vert f_{\rm sml}^{(r)} \right\vert (T_{1}^{n} x) \leq \frac{2 \| b\|_{\infty}}{r}.
\end{equation}
Since $r\in \mathbb{N}$ is arbitrarily large, $\limsup\limits_{M, N \rightarrow \infty} F_{M,N,b}(x) = 0$ on $X_{m}$. 
Hence, for $x \in X_{m}$, the averages in \eqref{eq:proof:average} converge for any $(m-1)$-step nilsequece $(b_{n})_{n}$.

Now we take $X_{f_{1}, \dots, f_{k}} = \bigcap_{m\in \mathbb{N}} X_{m}$. 
Then we obtain the desired result.
\end{proof}

%%%%%%%
We end this section by showing a result, which is a consequence of Proposition \ref{prop:conv_unif}. 
\begin{corollary} \label{cor:conv_weakmix}
Let ${\bm X} = (X, \mathcal{B}, \mu_{1}, \dots, \mu_{k}, T_{1}, \dots, T_{k})$ be a pointwise jointly ergodic system. 
Suppose that each $T_{i}$ is weakly mixing. 
Then for any $f_{1}, \dots, f_{k}\in L^{\infty}(X)$, there exists a full measure set $X_{0}\subset X$ such that 
\begin{equation*}
\lim_{N \rightarrow \infty} \frac{1}{N} \sum_{n=1}^{N} b_{n} \prod_{i=1}^{k} f_{i} (T_{i}^{n} x) =  \left[ \lim_{N \rightarrow \infty} \frac{1}{N} \sum_{n=1}^{N} b_{n}\right] \cdot \prod_{i=1}^{k} \int_{X} f_{i} \, d \mu_{i} 
\end{equation*}
for any $x \in X_{0}$ and for any nilsequence $(b_{n})_{n\in \mathbb{N}}$. 
\end{corollary}
\begin{proof}
Given an integer $m\geq 2$, let $(b_{n})_{n\in \mathbb{N}}$ be an $(m-1)$-step nilsequence. 
For $i = 1, 2, \dots, k$, define 
\[
f_{i}^{(0)} = f_{i} -  \int_{X} f_{i} \, d \mu_{i} \quad \text{and} \quad f_{i}^{(1)} = \int_{X} f_{i} \, d \mu_{i}. 
\]
Hence $f_{i} = f_{i}^{(0)} + f_{i}^{(1)}$ and $\int_{X} f_{i}^{(0)} \, d \mu_{i} = 0$ for every $i\in \{1,\dots,k\}$. 
We have 
\begin{equation*}
\frac{1}{N} \sum_{n=1}^{N} b_{n} \prod_{i=1}^{k} f_{i} (T_{i}^{n} x) = \frac{1}{N} \sum_{n=1}^{N} b_{n} \prod_{i=1}^{k} \int_{X} f_{i} \, d\mu_{i} + \sum_{ \substack {(j_{1}, \dots, j_{k}) \in \{0,1\}^{k} ; \\ (j_{1}, \dots, j_{k}) \ne (1,\dots,1)}} \frac{1}{N}\sum_{n=1}^{N} b_{n} \prod_{i=1}^{k} f_{i}^{(j_{i})} (T_{i}^{n} x).
\end{equation*}
Note here that each term $\prod_{i=1}^{k} f_{i}^{(j_{i})} (T_{i}^{n} x)$ in the sum above contains at least one $f_{i}^{(j_{i})}$ with $j_{i}=0$. 
Note also that for every $i\in \{1,\dots,k\}$ we have
\[
\left\VERT f_{i}^{(0)} \right\VERT_{m, T_{i}}=\left\vert \int_{X} f_{i}^{(0)} \, d \mu_{i}\right\vert =0
\]
since $T_{i}$ is weakly mixing. 
Hence it follows from Proposition \ref{prop:conv_unif} that there exists a full measure set $W_{m}\subset X$ such that for every $x\in W_{m}$
\[
\lim_{N \rightarrow \infty} \frac{1}{N} \sum_{n=1}^{N} b_{n} \prod_{i=1}^{k} f_{i} (T_{i}^{n} x) =  \left[ \lim_{N \rightarrow \infty} \frac{1}{N} \sum_{n=1}^{N} b_{n}\right] \cdot \prod_{i=1}^{k} \int_{X} f_{i} \, d \mu_{i}. 
\]
Letting $X_{0}=\cap_{m\in \mathbb{N}}W_{m}$, the proof of Corollary \ref{cor:conv_weakmix} is obtained. 
\end{proof}

%%%%%%%%%%%
\section{Multiple ergodic averages along the sequence $\lfloor \alpha n \rfloor$} 
\label{sec4:pf2}
 
In this section, we explore the limiting behavior of the following averages 
\begin{equation}
\label{eq:along:integer}
 \frac{1}{N} \sum_{n=1}^{N} f_{1} (T_{1}^{\lfloor \alpha n \rfloor} x) \cdots f_{k} (T_{k}^{\lfloor \alpha n \rfloor} x),
\end{equation}
where $\alpha \geq 1$ and $\lfloor \theta \rfloor$ is the greatest integer less than or equal to $\theta \in \mathbb{R}$.
As we will see, the averages \eqref{eq:along:integer} can be seen as the multiple ergodic averages weighted by a Besicovitch almost periodic sequence. 

In what follows, we first recall the notion of Besicovitch almost periodic function. 
Then, we will prove Theorem \ref{thm:multiple_bap} in Section \ref{subsec:proof_th1.5} and Theorem \ref{thm:equiv_multiple_bap} in Section \ref{subsec:proof_th1.6}, respectively.

%\subsubsection*{Notation and convention}

Throughout this section, let $\mathbbm{1}_{E}$ denote the indicator function of a set $E\subset \mathbb{R}$, that is,
\[
\mathbbm{1}_{E}(x)=
\begin{cases}
1, & x\in E, \\
0, & x\not\in E.
\end{cases}
\]
For $\theta \in \mathbb{R}$, we write $\mathrm{e}(\theta) = \exp{(2 \pi \sqrt{-1} \theta)}$ for brevity. 
%Let $\lfloor \theta \rfloor$ denote the greatest integer less than or equal to $\theta$, and $\{\theta\}=\theta -\lfloor \theta \rfloor$ represents the fractional part of $\theta$. 
Let $\{\theta\}=\theta -\lfloor \theta \rfloor$ represent the fractional part of $\theta$. 
By a slight abuse of notation, for $a\in [0,1)$, we write $\mathbbm{1}_{(a,1]}(\{\theta\})$ instead of $\mathbbm{1}_{\{0\}\cup (a,1)}(\{\theta\})$. 
In other words, the endpoint $1$ is identified with $0$ for the interval $[0,1)$, which is the range set of $\theta \mapsto \{\theta\}$. 

\subsubsection*{Besicovitch almost periodic sequences}

\begin{definition}
A function $f\colon \mathbb{N} \rightarrow \mathbb{C}$ is \emph{Besicovitch almost periodic} if for any $\epsilon > 0$, there exist  $c_{1}, \dots, c_{q}\in \mathbb{C}$ and $\beta_{1}, \dots, \beta_{q}\in \mathbb{R}$ such that 
\begin{equation}
\label{def:besicovitch_ap}
\limsup_{N \rightarrow \infty} \frac{1}{N} \sum_{n=1}^{N} \left| f(n) - \sum_{j=1}^{q} c_{j} \mathrm{e} \left( \beta_{j} n\right) \right| < \epsilon .
\end{equation}
\end{definition}

It is known that if $K$ is a compact Abelian group, $a \in K$ and $g\colon K \rightarrow \mathbb{C}$ is a Riemann integrable function, then $f(n):=g(an)$ is known to be Besicovitch almost periodic. 
(See \cite{Besicovitch1955} for more detail. 
See also \cite{Bellow-Losert1985}*{Section 3} and \cite{Bergelson-Moreira2016}*{Section 4}.)

For $\alpha \geq 1$, let $\beta = 1/\alpha$ and $A = \{ \lfloor \alpha m\rfloor \colon m \in \mathbb{Z} \}$. 
Given $n\in \mathbb{N}$, we see that $n=\lfloor \alpha m \rfloor$ for some $m \in \mathbb{Z}$ if and only if $ \beta n \leq m <  \beta n + \beta$, and hence we see that $n \in A$ if and only if $\beta n\pmod 1\in (1-\beta, 1]$. 
Thus, for every $n\in \mathbb{Z}$, we have
\begin{equation}\label{eq:A_beta}
\mathbbm{1}_{A} (n) = \mathbbm{1}_{(1-\beta, 1]} ( \{ \beta n \}).
\end{equation}
(Note that $1-\beta \in [0,1)$. We use the convention mentioned above on the right-hand side of \eqref{eq:A_beta}.)
%where $\{r\}=r-\lfloor r\rfloor$ represents the fractional part of $r \in \mathbb{R}$. 
%Here note that we identify the endpoints 0 and 1 for the interval $[0,1)$. 
%Indeed this follows from that $n=\lfloor \alpha m \rfloor$ for some $m \in \mathbb{Z}$ if and only if $ \beta n \leq m <  \beta n + \beta$, and hence we see that $n \in A$ if and only if $n \beta \pmod 1\in (1-\beta, 1]$. 
Therefore, the function $\mathbbm{1}_{A}$ is Besicovitch almost periodic. 
Later, in Lemma \ref{lem:Bap_modify}, we show that for $f=\mathbbm{1}_{A}$, each $\beta_{j}\in \mathbb{R}$ in \eqref{def:besicovitch_ap} can be taken from the set $\{ m/\alpha \colon m \in \mathbb{Z} \}$. 

%Now we will show that ergodic averages along the sequence $\lfloor \alpha n \rfloor$ converge almost everywhere for a pointwie joint ergodic system. 
%\begin{theorem} \label{thm:multiple_bap}
%Suppose that a joint measure-preserving system $ {\bm X} = (X, \mathcal{B}, \mu_{1}, \dots, \mu_{k}, T_{1}, \dots, T_{k})$ is pointwise jointly ergodic. 
%Let $\alpha \geq 1$. 
%For any bounded measurable functions $f_{1}, \dots, f_{k}\in L^{\infty}(X)$, the averages
%\[ 
%\frac{1}{N} \sum_{n=1}^{N}  f_{1}( T_{1}^{\lfloor \alpha n\rfloor} x) \cdots f_{k} (T_{k}^{\lfloor \alpha n\rfloor} x)
%\]
%converge for almost every $x \in X$.
%\end{theorem}

%We now consider ergodic averages of the form \eqref{eq:along:integer}. 
%Let $A = \{ \lfloor \alpha m\rfloor \colon m \in \mathbb{Z} \}$. 
Now, we see that the study of limiting behavior of averages \eqref{eq:along:integer} is reduced to the study of the weighted averages of the form
\begin{equation*}
\label{eq:along:integerset}
 \frac{1}{ N } \sum_{n=1}^{ N }  \mathbbm{1}_{A} (n) f_{1} (T_{1}^{n} x) \cdots f_{k} (T_{k}^{n} x)
 \end{equation*}
since, for any sequence $(F_{n})_{n \in \mathbb{N}}$, 
\begin{equation}
\label{eq:equality:seq:set}
 \lim_{N \rightarrow \infty} \frac{1}{N} \sum_{n=1}^{N} F_{\lfloor \alpha n \rfloor} = \lim_{N \rightarrow \infty} \frac{\lfloor N \alpha \rfloor}{N} \frac{1}{\lfloor N \alpha \rfloor} \sum_{n=1}^{\lfloor N \alpha \rfloor}  \mathbbm{1}_{A} (n) F_{n} 
 = \alpha \cdot \lim_{N \rightarrow \infty} \frac{1}{ N} \sum_{n=1}^{N}  \mathbbm{1}_{A} (n) F_{n}
 \end{equation}
if the last limit exists. 

\subsection{Proof of Theorem \ref{thm:multiple_bap}} \label{subsec:proof_th1.5}
Now we provide the proof of Theorem \ref{thm:multiple_bap}.
\begin{proof}[Proof of Theorem \ref{thm:multiple_bap}]
We assume, without loss of generality, that $\| f_{i} \|_{L^{\infty}} \leq 1$ for $i = 1, 2, \dots, k$. 
For brevity, write $F_{n}(x)=\prod_{i=1}^{k}f_{i}(T_{i}^{n} x)$. 
Note that $|F_{n}(x)| \leq 1$ for every $n\in \mathbb{N}$ and for almost every $x \in X$.
By \eqref{eq:equality:seq:set}, it is sufficient to show that there exists a full measure set $X_{0}\subset X$ such that for $x \in X_{0}$,  
\begin{equation}
\label{eq:goal_th:multiple_bap}
\limsup_{M, N \rightarrow \infty} \left\vert \frac{1}{N} \sum_{n=1}^{N} \mathbbm{1}_{A} (n) F_{n}(x) -  \frac{1}{M} \sum_{n=1}^{M} \mathbbm{1}_{A} (n)F_{n}(x) \right\vert =0,
\end{equation}
where $A = \{ \lfloor \alpha m\rfloor \colon m \in \mathbb{Z} \}$. 

Since $\mathbbm{1}_{A}$ is a Besicovitch almost periodic function, given $r \in \mathbb{N}$, there exists a trigonometric polynomial $P_{r} (n) = \sum\limits_{j=1}^{q} c_{j} \mathrm{e} \left( \beta_{j} n\right)$ such that 
\begin{equation} \label{ineq:approx_1/m}
\limsup_{N \rightarrow \infty} \frac{1}{N} \sum_{n=1}^{N} \left| \mathbbm{1}_{A}(n) - P_{r}(n) \right| < \frac{1}{r} .
\end{equation}

Recall that for every $\beta \in \mathbb{R}$ the sequence $(\mathrm{e} \left( \beta n\right))_{n}$ is an 1-step nilsequence.
Therefore, by Theorem \ref{th:mlww_pje_nil}, there exists a full measure set $X_{0}\subset X$ such that for any $\beta \in \mathbb{R}$ and for  $x \in X_{0}$, the averages 
\[
\frac{1}{N} \sum_{n=1}^{N} \mathrm{e} \left( \beta n\right) f_{1}(T_{1}^{n} x) \cdots f_{k} (T_{k}^{n} x),
\]
and thus 
\begin{equation}\label{cond:conv_beta}
\lim_{M, N \to \infty} \left\vert \frac{1}{N} \sum_{n=1}^{N} P_{r}(n) F_{n}(x) -  \frac{1}{M} \sum_{n=1}^{M} P_{r}(n) F_{n}(x) \right\vert = 0
\end{equation}
for every $x \in X_{0}$.     
Then notice that 
\begin{align}
\label{eq4.8}
&\left\vert \frac{1}{N} \sum_{n=1}^{N} \mathbbm{1}_{A} (n)F_{n}(x) -  \frac{1}{M} \sum_{n=1}^{M} \mathbbm{1}_{A} (n)F_{n}(x) \right\vert \notag \\
&\leq  \left\vert \frac{1}{N} \sum_{n=1}^{N} \mathbbm{1}_{A} (n)F_{n}(x) -  \frac{1}{N} \sum_{n=1}^{N} P_{r}(n)F_{n}(x) \right\vert + \left\vert \frac{1}{N} \sum_{n=1}^{N} P_{r}(n) F_{n}(x) -  \frac{1}{M} \sum_{n=1}^{M} P_{r}(n) F_{n}(x) \right\vert \notag \\
&\quad + \left\vert \frac{1}{M} \sum_{n=1}^{M} P_{r} (n)F_{n}(x) -  \frac{1}{M} \sum_{n=1}^{M} \mathbbm{1}_{A} (n)F_{n}(x) \right\vert \notag \\
&\leq \frac{1}{N} \sum_{n=1}^{N} \left\vert \mathbbm{1}_{A} (n)- P_{r}(n)\right\vert + \left\vert \frac{1}{N} \sum_{n=1}^{N} P_{r}(n) F_{n}(x) -  \frac{1}{M} \sum_{n=1}^{M} P_{r}(n) F_{n}(x) \right\vert + \frac{1}{M} \sum_{n=1}^{M} \left\vert P_{r} (n) -  \mathbbm{1}_{A} (n) \right\vert .
\end{align}
Then, for any $x \in X_{0}$, we have \eqref{eq:goal_th:multiple_bap} by \eqref{ineq:approx_1/m}, \eqref{cond:conv_beta}, and \eqref{eq4.8}. 
The proof of Theorem \ref{thm:multiple_bap} is obtained. 
\end{proof}

%Let $E(T)$ denote the set of eigenvalues of a measure-preserving transformation $T$. For an ergodic measure-preserving transformation $T$, $E(T)$ forms a countable multiplicative subgroup of $\mathbb{T}$. 
%For a joint measure-preserving system $ {\bm X} = (X, \mathcal{B}, \mu_{1}, \dots, \mu_{k}, T_{1}, \dots, T_{k})$, we define $E(T_{1}, \dots, T_{k})$ as 
%$$E(T_{1}, \dots, T_{k}) = \{ \lambda \in \mathbb{T} \colon \lambda = \prod_{i=1}^{k} \lambda_{i}, \ \lambda_{i} \in E(T_{i})\}.$$

%{\color{cyan} The previous statement is wrong. We have an example that $T_{1} x= x + \sqrt{2}$, $T_{2} x = 2x$, $\alpha = \sqrt{2}$.}

\subsection{Proof of Theorem \ref{thm:equiv_multiple_bap}}\label{subsec:proof_th1.6}

%\begin{theorem}
%\label{thm:equiv_multiple_bap}
%Suppose that ${\bm X} = (X, \mathcal{B}, \mu_{1}, \dots, \mu_{k}, T_{1}, \dots, T_{k})$ is pointwise jointly ergodic.
%Let $\alpha \geq 1$. 
%The following are equivalent.
%\begin{enumerate}[{\rm (1)}]
%\item $\{ \mathrm{e} (m/\alpha)\colon m \in \mathbb{Z} \} \bigcap E (T_{1}, \dots, T_{k}) = \{1\}$. 
%\item For any bounded measurable functions $f_{1}, \dots, f_{k}$, 
%\[ 
%\lim_{N \rightarrow \infty} \frac{1}{N} \sum_{n=1}^{N}  f_{1}(T_{1}^{\lfloor \alpha n\rfloor} x) \cdots f_{k} (T_{k}^{\lfloor \alpha n\rfloor} x) = \prod_{i=1}^{k} \int_{X} f_{i} \, d\mu_{i}
%\]
%for almost every $x\in X$.
%\end{enumerate}
%\end{theorem}

Before embarking on the proof of Theorem \ref{thm:equiv_multiple_bap}, we need to introduce three lemmas. 
\begin{lemma}\label{lem:unif_distr}
Let $\alpha$ be a non-zero real number. 
Then, for any $m \in \mathbb{Z}$
\[ 
\lim_{N \rightarrow \infty} \frac{1}{N} \sum_{n=1}^{N} \mathrm{e} \left( \lfloor \alpha n \rfloor \frac{m}{\alpha}\right) \ne 0.
\]
\end{lemma}

\begin{proof}
If $\alpha \in \mathbb{Z}$, then $\mathrm{e} \left( \lfloor \alpha n \rfloor \frac{m}{\alpha}\right) = 1$, and hence $\lim\limits_{N \rightarrow \infty} \frac{1}{N} \sum\limits_{n=1}^{N} \mathrm{e} \left( \lfloor \alpha n \rfloor \frac{m}{\alpha}\right) = 1$. 
If $\alpha \in \mathbb{Q} \setminus \mathbb{Z}$, let $\alpha = p/q$, where $p, q \in \mathbb{Z}$ with $q \geq 2$ such that $(p, q) =1$. 
Then one has 
\[
\mathrm{e} \left( \lfloor \alpha n \rfloor \frac{m}{\alpha}\right)=\mathrm{e} \left( \left( \alpha n- \{\alpha n\} \right) \frac{m}{\alpha} \right)=\mathrm{e} \left( - \{\alpha n\} \frac{m}{\alpha} \right)=\mathrm{e} \left( - \left\{ \frac{p}{q} n\right\} \frac{qm}{p} \right) .
\]
Since the sequence $\left( \left\{ (p/q) n \right\} \right)_{n \in \mathbb{N}}$ is equidistributed on the set $\{ j/q\colon j = 0, 1, \dots q-1 \}$, we have
\[
\lim_{N \rightarrow \infty} \frac{1}{N} \sum_{n=1}^{N} \mathrm{e} \left( \lfloor \alpha n \rfloor \frac{m}{\alpha}\right)
= \lim_{N \rightarrow \infty} \frac{1}{N} \sum_{n=1}^{N} \mathrm{e} \left( - \left\{ \frac{p}{q} n\right\} \frac{qm}{p} \right)  
= \frac{1}{q} \sum_{j=0}^{q-1} \mathrm{e} \left( - \frac{mj}{p}\right) \ne 0.
\]

Finally, suppose $\alpha \in \mathbb{R} \setminus \mathbb{Q}$. 
Let $\beta =m/\alpha$. 
Then one has
\[ 
\lim_{N \rightarrow \infty} \frac{1}{N} \sum_{n=1}^{N} \mathrm{e} \left( \lfloor \alpha n \rfloor \beta \right) 
= \lim_{N \rightarrow \infty} \frac{1}{N} \sum_{n=1}^{N} \mathrm{e} \left( - \{\alpha n\} \beta \right)  
= \int_{0}^{1} \mathrm{e} \left( \beta x \right) \, dx  \ne 0
\]
since the sequence $(\alpha n)_{n \in \mathbb{N}}$ is uniformly distributed mod $1$.
\end{proof}

As was mentioned, for a Besicovitch almost periodic sequence  $\mathbbm{1}_{A}$, where $A = \{ \lfloor \alpha m\rfloor \colon m \in \mathbb{Z} \}$, each $\beta_{j}\in \mathbb{R}$ in \eqref{def:besicovitch_ap} can be taken from the set $\{ m/\alpha \colon m \in \mathbb{Z} \}$.
%Therefore, $\mathbbm{1}_{A}$ is Besicovitch almost periodic: for any $\epsilon > 0$, there exist  $c_{1}, \dots, c_{q}\in \mathbb{C}$ and $\beta_{1}, \dots, \beta_{q}\in \mathbb{R}$ such that 
%\begin{equation}
%\label{eqn:besicovitch_ap}
%\limsup_{N \rightarrow \infty} \frac{1}{N} \sum_{n=1}^{N} \left| \mathbbm{1}_{A}(n) - %\sum_{j=1}^{q} c_{j} \mathrm{e} \left( \beta_{j} n\right) \right| < \epsilon .
%\end{equation}
%In addition, for  $\mathbbm{1}_{A}$, one can take each $\beta_{j}\in \mathbb{R}$ in \eqref{eqn:besicovitch_ap} from the set $\{ m/\alpha \colon m \in \mathbb{Z} \}$; namely we have the following result. 

%We provide a proof for reader's convenience. {\color{red} delete this sentence?}
\begin{lemma} \label{lem:Bap_modify}
For $\alpha \geq 1$, let $\beta = 1/\alpha$ and $A = \{ \lfloor \alpha m\rfloor \colon m \in \mathbb{Z} \}$. 
For any $\epsilon > 0$, there exist $c_{1}, \dots, c_{q}\in \mathbb{C}$ and $m_{1}, \dots, m_{q}\in \mathbb{Z}$ such that 
\begin{equation*}
\limsup_{N \rightarrow \infty} \frac{1}{N} \sum_{n=1}^{N} \left| \mathbbm{1}_{A}(n) - \sum_{j=1}^{q} c_{j} \mathrm{e} \left(  m_{j} \beta n\right) \right| < \epsilon .
\end{equation*}
\end{lemma}
In the following proof, we identify $\mathbb{T} = \mathbb{R}/ \mathbb{Z}$. 
\begin{proof}
Let us first consider $\alpha \in \mathbb{Q}$. 
Write $\alpha = a/b$, where $a, b \in \mathbb{N}$ and $(a,b) =1$. 
Note that 
\[
n = \left\lfloor \frac{a}{b} m \right\rfloor \quad \text{if and only if} \quad n + a = \left\lfloor \frac{a}{b} (m+b) \right\rfloor,
\]
hence one has $\mathbbm{1}_{A} (n + a) = \mathbbm{1}_{A} (n)$ for any $n\in \mathbb{N}$. 
Therefore, there exist $c_{1}, \dots, c_{a} \in \mathbb{C}$ such that 
\begin{equation*}
\mathbbm{1}_{A}(n) = \sum_{j=1}^{a} c_{j} \mathrm{e} \left( \frac{j-1}{a} n \right) .
\end{equation*}
For each $j = 1, \dots, a$, one can find $m_{j} \in \mathbb{Z}$ such that $(j-1)/a \equiv m_{j} b/a \pmod a$. 
Hence it follows that
\begin{equation*}
\mathbbm{1}_{A}(n) = \sum_{j=1}^{a} c_{j} \mathrm{e} \left( m_{j}\beta n \right) .
\end{equation*}

Now consider $\alpha \notin \mathbb{Q}$. By \eqref{eq:A_beta}, we will consider $\mathbbm{1}_{(1-\beta, 1]} ( \{ \beta n \})$ instead of $\mathbbm{1}_{A} (n)$. 
For the indicator function $\mathbbm{1}_{(1-\beta, 1]}\colon \mathbb{T}\to \{ 0,1\}$, one can find a function $f \in C^{2} (\mathbb{T})$ such that  $|f (x)| \leq 1$ for all $x\in \mathbb{T}$ and $I = \{ x\in \mathbb{T} \colon  f(x) \ne \mathbbm{1}_{(1-\beta, 1]} (x) \}$ is an open interval with ${\rm Leb}_{\mathbb{T}} (I) < \epsilon /3$, where ${\rm Leb}_{\mathbb{T}}$ is the Lebesgue measure on $\mathbb{T}$. 
Since $f \in C^{2} (\mathbb{T})$, its Fourier coefficients $\hat{f} (n) = O (1/n^{2})$ as $\vert n\vert \to \infty$, and hence the Fourier series of $f$ converge to $f$ uniformly. 
Thus, there exists $c_{1}, \dots, c_{q} \in \mathbb{C}$ and $m_{1}, \dots, m_{q} \in \mathbb{Z}$ such that
\begin{equation} \label{ineq:approx}
\left| f(x) - \sum_{j=1}^{q} c_{j} \mathrm{e} \left( m_{j} x\right) \right| < \epsilon /3
\end{equation}
for every $x\in \mathbb{T}$. 
It follows from \eqref{eq:A_beta} and \eqref{ineq:approx} that 
\begin{align*}
\left| \mathbbm{1}_{A} (n) - \sum_{j=1}^{q} c_{j} \mathrm{e} \left( m_{j} \beta n\right) \right| 
&= \left| \mathbbm{1}_{(1-\beta, 1]} ( \{n \beta \}) - \sum_{j=1}^{q} c_{j} \mathrm{e} \left( m_{j} \beta n\right) \right| \\ 
&\leq \left| \mathbbm{1}_{(1-\beta, 1]} ( \{n \beta \}) - f ( \{n \beta \}) \right| +\left| f ( \{n \beta \}) - \sum_{j=1}^{q} c_{j} \mathrm{e} \left( m_{j} \beta n\right) \right| \\ 
&< \left| \mathbbm{1}_{(1-\beta, 1]} ( \{n \beta \}) - f ( \{n \beta \}) \right| +\frac{\epsilon}{3} .
\end{align*}
(Note that $\mathrm{e} \left( m \{r\}\right)=\mathrm{e} \left( m r\right)$ for every $m\in \mathbb{Z}$ and $r\in \mathbb{R}$.) 
Therefore, one has 
\begin{align*}
\frac{1}{N} \sum_{n=1}^{N} \left| \mathbbm{1}_{A} (n) - \sum_{j=1}^{q} c_{j} \mathrm{e} \left( m_{j} \beta n\right) \right| 
&< \frac{1}{N} \sum_{n=1}^{N}\left| \mathbbm{1}_{(1-\beta, 1]} ( \{n \beta \}) - f ( \{n \beta \}) \right| +\frac{\epsilon}{3} .
\end{align*}
%Note here that, since $\mathbbm{1}_{(1-\beta, 1]} = f $ on $\mathbb{T}\setminus I$, we have
%\begin{align*}
%\frac{1}{N} \sum_{n=1}^{N}\left| \mathbbm{1}_{(1-\beta, 1]} ( \{n \beta \}) - f ( \{n \beta \}) \right| 
%&=\frac{1}{N} \sum_{\substack{n=1 \\ \{n \beta \}\in I}}^{N}\left| \mathbbm{1}_{(1-\beta, 1]} ( \{n \beta \}) - f ( \{n \beta \}) \right| \\
%&\leq \frac{1}{N} \sum_{\substack{n=1 \\ \{n \beta \}\in I}}^{N} \left( \left| \mathbbm{1}_{(1-\beta, 1]} ( \{n \beta \})\right| + \left| f ( \{n \beta \}) \right| \right) \\
%&\leq \frac{2}{N} \sum_{n=1}^{N} \mathbbm{1}_{I} ( \{n \beta \}). \end{align*}
%{\color{cyan} What about the following instead of the previous one?}
Note here that, since $|\mathbbm{1}_{(1-\beta, 1]} - f| \leq 2 $ on $\mathbb{T}$ and $\mathbbm{1}_{(1-\beta, 1]} = f $ on $\mathbb{T}\setminus I$, we have
\begin{equation*}
\frac{1}{N} \sum_{n=1}^{N}\left| \mathbbm{1}_{(1-\beta, 1]} ( \{n \beta \}) - f ( \{n \beta \}) \right| \leq \frac{2}{N} \sum_{n=1}^{N} \mathbbm{1}_{I} ( \{n \beta \}). 
\end{equation*}
Thus it follows that 
\begin{align*}
\limsup_{N \rightarrow \infty} \frac{1}{N} \sum_{n=1}^{N} \left| \mathbbm{1}_{A} (n) - \sum_{j=1}^{q} c_{j} \mathrm{e} \left( m_{j} \beta n\right) \right| \leq  \lim_{N \rightarrow \infty} \frac{2}{N} \sum_{n=1}^{N} \mathbbm{1}_{I} (\{n \beta\}) + \frac{\epsilon}{3}  =2{\rm Leb}_{\mathbb{T}} (I) +\frac{\epsilon}{3}  < \epsilon,
\end{align*}
since $(\beta n)_{n \in \mathbb{N}}$ is uniformly distributed $\bmod \, 1$.
\end{proof}

\begin{lemma}
\label{lem:conv:mostlambda}
Let ${\bm X} = (X, \mathcal{B}, \mu_{1}, \dots, \mu_{k}, T_{1}, \dots, T_{k})$ be a pointwise jointly ergodic system. 
For any bounded measurable functions $f_{1}, \dots, f_{k}$ on $X$, there exists a full measure set $X_{0}\subset X$ such that if $\lambda \notin  E (T_{1}, \dots, T_{k})$ and $x \in X_{0}$, then   
\[ 
\lim_{N \rightarrow \infty} \frac{1}{N} \sum_{n=1}^{N} \lambda^{n} f_{1}(T_{1}^{n} x) \cdots f_{k} (T_{k}^{n} x) = 0.
\]
\end{lemma}

\begin{proof}
For $h \in \mathbb{N}$, write $f_{i,h} (x) = f_{i} (T_{i}^{h} x) \overline{f_{i}(x)}$ for $i =1, 2, \dots, k$.
Since ${\bm X}$ is pointwise jointly ergodic, there is a full measure set  $X_{0}\subset X$ such that if $x \in X_{0}$, then
\begin{equation}\label{eqn:pje}
\lim_{N \rightarrow \infty} \frac{1}{N} \sum_{n=1}^{N} f_{1,h} (T_{1}^{n} x) \cdots f_{k,h} (T_{k}^{n} x) = \prod_{i=1}^{k} \int_{X}  f_{i} \circ T_{i}^{h} \cdot \overline{f_{i}} \, d \mu_{i}
\end{equation}
for any $h \in \mathbb{N}$. 

We will apply van der Corput trick for $x_{n} = \lambda^{n} f_{1} (T_{1}^{n} x) \cdots f_{k} (T_{k}^{n} x)$, where $x\in X_{0}$. 
Since 
\[
\langle x_{n+h}, x_{n} \rangle_{\mathbb{C}} = \lambda^{h} \prod_{i=1}^{k} f_{i} (T_{i}^{n+h} x) \overline{f_{i}(T_{i}^{n} x)} = \lambda^{h} \prod_{i=1}^{k} f_{i,h} (T_{i}^{n} x), 
\]
in view of \eqref{eqn:pje} we have 
\begin{equation} \label{eq:vdc0_mostlambda}
\lim_{N \rightarrow \infty} \frac{1}{N} \sum_{n=1}^{N} \langle x_{n+h}, x_{n} \rangle_{\mathbb{C}} = \lim_{N \rightarrow \infty} \frac{\lambda^{h}}{N} \sum_{n=1}^{N} \prod_{i=1}^{k} f_{i,h} (T_{i}^{n} x) = \lambda^{h} \prod_{i=1}^{k} \int_{X} f_{i} \circ T_{i}^{h} \cdot \overline{f_{i}} \, d \mu_{i}. 
\end{equation}
Consider the product system $(X^{k}, S, \nu)$, where $S = T_{1} \times \cdots \times T_{k}\colon X^{k}\to X^{k}$ and $\nu =\otimes_{i=1}^{k} \mu_{i}$. 
Let $F = f_{1} \otimes \cdots \otimes f_{k}\in L^{\infty}(X^{k},\nu)$. 
By \eqref{eq:vdc0_mostlambda}, we have
\begin{equation}\label{eq:vdc1_mostlambda}
\lim_{N \rightarrow \infty} \frac{1}{N} \sum_{n=1}^{N} \langle x_{n+h}, x_{n} \rangle_{\mathbb{C}} =\lambda^{h} \int_{X^{k}}  F \circ S^{h} \cdot \overline{F} \, d \nu =\lambda^{h} \langle F \circ S^{h},F \rangle_{L^{2}(X^{k},\nu)}. 
\end{equation}
Since $E(S) = E(T_{1}, \dots, T_{k})$, one has $\lambda \notin E (S)$ by assumption.  
Define an isometry $U$ on $L^{2}(X^{k}, \nu )$ by $U g = \lambda \cdot g \circ S$. 
Let $I\subset L^{2}(X^{k},\nu)$ be the $U$-invariant subspace, that is, 
\[
I=\left\{g\in L^{2}(X^{k},\nu)\colon Ug = g\right\} = \left\{g\in L^{2}(X^{k},\nu)\colon g\circ S = \overline{\lambda} g\right\}.
\]
Note that $I$ is trivial, since $\lambda \notin E (S)$.
Therefore, by the von Neumann mean ergodic theorem, we have 
\begin{equation}\label{eq:vdc2_mostlambda}
\lim_{H \rightarrow \infty} \frac{1}{H} \sum_{h=1}^{N} \lambda^{h} \cdot F\circ S^{h}=  \lim_{H \rightarrow \infty} \frac{1}{H} \sum_{h=1}^{N} U^{h} F   =  0
\end{equation}
in $L^{2}(X^{k},\nu)$, and hence, in view of \eqref{eq:vdc1_mostlambda} and \eqref{eq:vdc2_mostlambda},
\[ 
\lim_{H \rightarrow \infty} \frac{1}{H} \sum_{h=1}^{H}\lim_{N \rightarrow \infty} \frac{1}{N} \sum_{n=1}^{N} \langle x_{n+h}, x_{n} \rangle_{\mathbb{C}} = \lim_{H \rightarrow \infty} \frac{1}{H} \sum_{h=1}^H \lambda^{h} \langle F \circ S^{h},F \rangle_{L^{2}(X^{k},\nu)} = 0.
\]
Then the desired result follows from Lemma \ref{lem:vdC}. 
\end{proof}

Now we present the proof of Theorem \ref{thm:equiv_multiple_bap}.
\begin{proof}[Proof of Theorem \ref{thm:equiv_multiple_bap}]
 $(1) \Rightarrow (2)$: 
First, we will show that (1) implies that if $\int_{X} f_{i} \, d \mu_{i} = 0$ for some $i\in \{ 1, 2, \dots, k\}$, then 
\[ 
\lim_{N \rightarrow \infty} \frac{1}{N} \sum_{n=1}^{N} f_{1}(T_{1}^{\lfloor \alpha n\rfloor} x) \cdots f_{k} (T_{k}^{\lfloor \alpha n\rfloor} x) = 0,
\]
or equivalently
\[ 
\lim_{N \rightarrow \infty} \frac{1}{N} \sum_{n=1}^{N} \mathbbm{1}_{A}(n)  f_{1}(T_{1}^{n} x) \cdots f_{k} (T_{k}^{n} x) = 0
\]
for almost every $x\in X$, where $A = \{ \lfloor \alpha m\rfloor \colon m \in \mathbb{Z} \}$. 
Let $\beta =1/\alpha$. 
By Lemma \ref{lem:Bap_modify}, for any $\epsilon > 0 $, there exist a trigonometric polynomial $P (n) = \sum\limits_{j=1}^{q} c_{j} \mathrm{e} \left( m_{j} \beta n \right)$ for some $c_{1}, \dots, c_{q}\in \mathbb{C}$ and $m_{1}, \dots, m_{q}\in \mathbb{Z}$ such that 
\[
\limsup\limits_{N \rightarrow \infty} \frac{1}{N} \sum\limits_{n=1}^{N} \left| \mathbbm{1}_{A}(n) - P(n) \right| < \epsilon.
\]
Then it follows that
\[ 
\limsup\limits_{N \rightarrow \infty} \left\vert \frac{1}{N} \sum_{n=1}^{N} \mathbbm{1}_{A}(n)  f_{1}(T_{1}^{n} x) \cdots f_{k} (T_{k}^{n} x)  -  \frac{1}{N} \sum_{n=1}^{N} P(n)  f_{1}(T_{1}^{n} x) \cdots f_{k} (T_{k}^{n} x) \right\vert \leq \epsilon \cdot \prod_{i=1}^{k} \| f_{i}\|_{L^{\infty}}.
\]
Therefore, it remains to show that for any $m \in \mathbb{Z}$, 
\begin{equation}\label{eq:zero_average}
\lim_{N \rightarrow \infty} \frac{1}{N} \sum_{n=1}^{N} \mathrm{e} \left( m\beta n\right) f_{1}(T_{1}^{n} x) \cdots f_{k} (T_{k}^{n} x) = 0
\end{equation}
for almost every $x\in X$. 

If $\mathrm{e} \left( m \beta \right) \in E (T_{1}, \dots, T_{k})$, then $\mathrm{e} \left( m \beta \right)=1$ by the assumption (1), and hence \eqref{eq:zero_average} follows since ${\bm X}$ is pointwise jointly ergodic and $\int_{X} f_{i} \, d \mu_{i} = 0$ for some $i\in \{ 1, 2, \dots, k\}$.
Otherwise, we have \eqref{eq:zero_average} by Lemma \ref{lem:conv:mostlambda}.

Now we consider the general case. 
If we write 
$f_{i}^{(0)} = f_{i} -  \int_{X} f_{i} \, d \mu_{i}$ and $f_{i}^{(1)} = \int_{X} f_{i} \, d \mu_{i}$  for $i = 1, 2, \dots, k$, then
\begin{align}
&\frac{1}{N} \sum_{n=1}^{N}  f_{1}(T_{1}^{\lfloor \alpha n\rfloor} x) \cdots f_{k} (T_{k}^{\lfloor \alpha n\rfloor} x) \notag \\
%&\quad = \frac{1}{N}\sum_{n=1}^{N} \prod_{i=1}^{k} f_{i}^{(1)} (T_{i}^{\lfloor \alpha n\rfloor} x) + \sum_{ \substack {(j_{1}, \dots, j_{k}) \in \{1,2\}^{k} ; \\ (j_{1}, \dots, j_{k}) \ne (1,\dots,1)}} \frac{1}{N}\sum_{n=1}^{N}  f_{1}^{(j_{1})} (T_{1}^{\lfloor \alpha n\rfloor} x) \cdots f_{k}^{(j_{k})} (T_{k}^{\lfloor \alpha n\rfloor} x) \notag \\
&\quad = \prod_{i=1}^{k} \int_{X} f_{i}\, d \mu_{i} + \sum_{ \substack {(j_{1}, \dots, j_{k}) \in \{0,1\}^{k} ; \\ (j_{1}, \dots, j_{k}) \ne (1,\dots,1)}} \frac{1}{N}\sum_{n=1}^{N}  f_{1}^{(j_{1})} (T_{1}^{\lfloor \alpha n\rfloor} x) \cdots f_{k}^{(j_{k})} (T_{k}^{\lfloor \alpha n\rfloor} x). \label{eq:decomp_zeroaverage}
\end{align}
Here note that each term $f_{1}^{(j_{1})} (T_{1}^{\lfloor \alpha n\rfloor} x) \cdots f_{k}^{(j_{k})} (T_{k}^{\lfloor \alpha n\rfloor} x)$ in the sum on the second term of \eqref{eq:decomp_zeroaverage} contains at least one $f_{i}^{(j_{i})}$ with $j_{i}=0$. 
Note that $\int_{X} f_{i}^{(0)} \, d \mu_{i} = 0$, hence,  for each $(j_{1}, \dots, j_{k}) \in \{0,1\}^{k}\setminus (1,\dots,1)$, we have 
\[
\lim_{N\to \infty} \frac{1}{N}\sum_{n=1}^{N}  f_{1}^{(j_{1})} (T_{1}^{\lfloor \alpha n\rfloor} x) \cdots f_{k}^{(j_{k})} (T_{k}^{\lfloor \alpha n\rfloor} x)=0
\]
for almost every $x\in X$ as was observed. 
In view of \eqref{eq:decomp_zeroaverage} we have (2).

$(2) \Rightarrow (1)$:
Suppose that (1) does not hold. 
Hence, there exist $m \in \mathbb{Z}\setminus \{0\}$, $\beta_{1}, \dots, \beta_{k} \in \mathbb{R}$ with $\sum_{i=1}^{k} \beta_{i} = m/\alpha \ne 0 \pmod 1$, and eigenfunctions $f_{1}, \dots, f_{k}$ such that $f_{i}(T_{i} x) = \mathrm{e} \left( \beta_{i}\right) f_{i}(x)$ for almost every $x\in X$. 
Note that for every $i\in \{1,2,\dots ,k\}$, the system $(X, \mathcal{B}, \mu, T_{i})$ is ergodic, hence we can assume $\vert f_{i}(x)\vert =1$ for almost every $x\in X$. 
Since $\sum_{i=1}^{k} \beta_{i} \ne 0 \pmod 1$, one has $\mathrm{e} \left( \beta_{i}\right) \ne 1$ for some $i\in \{1,\dots,k\}$, and thus  $\int_{X} f_{i}\, d\mu_{i} = 0$. 
Hence, by (2), 
\[ 
\lim_{N \rightarrow \infty} \frac{1}{N} \sum_{n=1}^{N}  f_{1}(T_{1}^{\lfloor \alpha n\rfloor} x) \cdots f_{k} (T_{k}^{\lfloor \alpha n\rfloor} x) = \prod_{i=1}^{k} \int_{X} f_{i} \,d\mu_{i} = 0 
\]
for almost every $x\in X$. 
However, since 
\[
\prod_{i=1}^{k} f_{i} (T_{i}^{\lfloor \alpha n\rfloor}x) = \prod_{i=1}^{k} \mathrm{e} \left( \lfloor \alpha n\rfloor \beta_{i}\right) f_{i} (x) = \mathrm{e} \left( \lfloor \alpha n\rfloor \frac{m}{\alpha}  \right) \cdot \prod_{i=1}^{k} f_{i}(x)
\]
and $\prod_{i=1}^{k} f_{i}(x)\neq 0$ for almost every $x\in X$, 
\[ 
\lim_{N \rightarrow \infty} \frac{1}{N} \sum_{n=1}^{N} \prod_{i=1}^{k} f_{i} (T_{i}^{\lfloor \alpha n\rfloor}x) 
= \left[ \lim_{N \rightarrow \infty} \frac{1}{N} \sum_{n=1}^{N} \mathrm{e} \left( \lfloor \alpha n\rfloor \frac{m}{\alpha}  \right) \right] \cdot \prod_{i=1}^{k} f_{i}(x)  \ne 0
%\rightarrow \int_{0}^{1} e^{2 \pi \sqrt{-1} \frac{1}{\alpha} t} \, dt \cdot \prod_{j=1}^{k} f_{i}(x)
\] 
by Lemma \ref{lem:unif_distr}, and we have a contradiction. 
\end{proof}

\section*{Acknowledgments} 
MH was partially supported by the Japan Society for the Promotion of Science (JSPS) KAKENHI Grant Number 19K03558. 
YS was supported by the National Research Foundation of Korea (NRF grant number: 2020R1A2C1A01005446) and Basic Science Research Institute Fund (NRF grant number: 2021R1A6A1A10042944).


\begin{bibdiv}
\begin{biblist}

\begin{comment}
\bib{Assani2003}{book}{
   author={Assani, Idris},
   title={Wiener Wintner ergodic theorems},
   publisher={World Scientific Publishing Co., Inc., River Edge, NJ},
   date={2003},
   pages={xii+216},
   isbn={981-02-4439-8},
   review={\MR{1995517}},
   doi={10.1142/4538},
}
\end{comment}

\bib{Assani-Moore2018}{article}{
   author={Assani, Idris},
   author={Moore, Ryo},
   title={Extension of Wiener-Wintner double recurrence theorem to
   polynomials},
   journal={J. Anal. Math.},
   volume={134},
   date={2018},
   number={2},
   pages={597--613},
   issn={0021-7670},
   review={\MR{3771493}},
   doi={10.1007/s11854-018-0019-x},
}

\begin{comment}
\bib{Auslander-Green-Hahn1963}{book}{
   author={Auslander, L.},
   author={Green, L.},
   author={Hahn, F.},
   title={Flows on homogeneous spaces},
   series={Annals of Mathematics Studies},
   volume={No. 53},
   note={With the assistance of L. Markus and W. Massey, and an appendix by
   L. Greenberg},
   publisher={Princeton University Press, Princeton, NJ},
   date={1963},
   pages={vii+107},
   review={\MR{0167569}},
}
\end{comment}

\bib{Bellow-Losert1985}{article}{
   author={Bellow, A.},
   author={Losert, V.},
   title={The weighted pointwise ergodic theorem and the individual ergodic
   theorem along subsequences},
   journal={Trans. Amer. Math. Soc.},
   volume={288},
   date={1985},
   number={1},
   pages={307--345},
   issn={0002-9947},
   review={\MR{0773063}},
   doi={10.2307/2000442},
}

\bib{Berend1985}{article}{
   author={Berend, Daniel},
   title={Joint ergodicity and mixing},
   journal={J. Analyse Math.},
   volume={45},
   date={1985},
   pages={255--284},
   issn={0021-7670},
   review={\MR{833414}},
   doi={10.1007/BF02792552},
}

\bib{Berend-Bergelson1984}{article}{
   author={Berend, Daniel},
   author={Bergelson, Vitaly},
   title={Jointly ergodic measure-preserving transformations},
   journal={Israel J. Math.},
   volume={49},
   date={1984},
   number={4},
   pages={307--314},
   issn={0021-2172},
   review={\MR{788255}},
   doi={10.1007/BF02760955},
}

\bib{Berend-Bergelson1986}{article}{
   author={Berend, Daniel},
   author={Bergelson, Vitaly},
   title={Characterization of joint ergodicity for noncommuting
   transformations},
   journal={Israel J. Math.},
   volume={56},
   date={1986},
   number={1},
   pages={123--128},
   issn={0021-2172},
   review={\MR{879919}},
   doi={10.1007/BF02776245},
}

\begin{comment}
\bib{Bergelson-Leibman2002}{article}{
   author={Bergelson, V.},
   author={Leibman, A.},
   title={A nilpotent Roth theorem},
   journal={Invent. Math.},
   volume={147},
   date={2002},
   number={2},
   pages={429--470},
   issn={0020-9910},
   review={\MR{1881925}},
   doi={10.1007/s002220100179},
}
\end{comment}

\bib{Bergelson-Moreira2016}{article}{
   author={Bergelson, Vitaly},
   author={Moreira, Joel},
   title={Van der Corput's difference theorem: some modern developments},
   journal={Indag. Math. (N.S.)},
   volume={27},
   date={2016},
   number={2},
   pages={437--479},
   issn={0019-3577},
   review={\MR{3479166}},
   doi={10.1016/j.indag.2015.10.014},
}
		
\bib{Bergelson-Son2023}{article}{
   author={Bergelson, Vitaly},
   author={Son, Younghwan},
   title={Joint ergodicity of piecewise monotone interval maps},
   journal={Nonlinearity},
   volume={36},
   date={2023},
   number={6},
   pages={3376--3418},
   issn={0951-7715},
   review={\MR{4594746}},
}
	

\bib{Bergelson-Son2022}{article}{
   author={Bergelson, Vitaly},
   author={Son, Youngwan},
   title={Joint normality of representations of numbers: an ergodic approach},
 journal = {accepted for publication in Ann. Sc. Norm. Super. Pisa Cl. Sci. (5)}
}

\bib{Besicovitch1955}{book}{
   author={Besicovitch, A. S.},
   title={Almost periodic functions},
   publisher={Dover Publications, Inc., New York},
   date={1955},
   pages={xiii+180},
   review={\MR{0068029}},
}

\begin{comment}
    \bib{Bourgain1989}{article}{
   author={Bourgain, Jean},
   title={Pointwise ergodic theorems for arithmetic sets},
   note={With an appendix by the author, Harry Furstenberg, Yitzhak
   Katznelson and Donald S. Ornstein},
   journal={Inst. Hautes \'{E}tudes Sci. Publ. Math.},
   number={69},
   date={1989},
   pages={5--45},
   issn={0073-8301},
   review={\MR{1019960}},
}
\end{comment}


\bib{Bourgain1990}{article}{
   author={Bourgain, J.},
   title={Double recurrence and almost sure convergence},
   journal={J. Reine Angew. Math.},
   volume={404},
   date={1990},
   pages={140--161},
   issn={0075-4102},
   review={\MR{1037434}},
   doi={10.1515/crll.1990.404.140},
}




\bib{Chu2009}{article}{
   author={Chu, Qing},
   title={Convergence of weighted polynomial multiple ergodic averages},
   journal={Proc. Amer. Math. Soc.},
   volume={137},
   date={2009},
   number={4},
   pages={1363--1369},
   issn={0002-9939},
   review={\MR{2465660}},
   doi={10.1090/S0002-9939-08-09614-7},
}

\bib{Derrien-Lesigne1996}{article}{
   author={Derrien, Jean-Marc},
   author={Lesigne, Emmanuel},
   title={Un th\'{e}or\`eme ergodique polynomial ponctuel pour les
   endomorphismes exacts et les $K$-syst\`emes},
   language={French, with English and French summaries},
   journal={Ann. Inst. H. Poincar\'{e} Probab. Statist.},
   volume={32},
   date={1996},
   number={6},
   pages={765--778},
   issn={0246-0203},
   review={\MR{1422310}},
}

\begin{comment}
\bib{Frantzikinakis2006}{article}{
   author={Frantzikinakis, Nikos},
   title={Uniformity in the polynomial Wiener-Wintner theorem},
   journal={Ergodic Theory Dynam. Systems},
   volume={26},
   date={2006},
   number={4},
   pages={1061--1071},
   issn={0143-3857},
   review={\MR{2246591}},
   doi={10.1017/S0143385706000204},
}
\end{comment}

\bib{Furstenberg1967}{article}{
   author={Furstenberg, Harry},
   title={Disjointness in ergodic theory, minimal sets, and a problem in
   Diophantine approximation},
   journal={Math. Systems Theory},
   volume={1},
   date={1967},
   pages={1--49},
   issn={0025-5661},
   review={\MR{0213508}},
   doi={10.1007/BF01692494},
}

\bib{Furstenberg1981}{book}{
   author={Furstenberg, H.},
   title={Recurrence in ergodic theory and combinatorial number theory},
   note={M. B. Porter Lectures},
   publisher={Princeton University Press, Princeton, NJ},
   date={1981},
   pages={xi+203},
   isbn={0-691-08269-3},
   review={\MR{0603625}},
}

\begin{comment}
\bib{HKS2022}{article}{
   author={Hirayama, Michihiro},
   author={Kim, Dong Han},
   author={Son, Younghwan},
   title={On the multiple recurrence properties for disjoint systems},
   journal={Israel J. Math.},
   volume={247},
   date={2022},
   number={1},
   pages={405--431},
   issn={0021-2172},
   review={\MR{4425342}},
   doi={10.1007/s11856-021-2271-5},
}
\end{comment}

\bib{Host-Kra2005}{article}{
   author={Host, Bernard},
   author={Kra, Bryna},
   title={Nonconventional ergodic averages and nilmanifolds},
   journal={Ann. of Math. (2)},
   volume={161},
   date={2005},
   number={1},
   pages={397--488},
   issn={0003-486X},
   review={\MR{2150389}},
   doi={10.4007/annals.2005.161.397},
}

\bib{Host-Kra2009}{article}{
   author={Host, Bernard},
   author={Kra, Bryna},
   title={Uniformity seminorms on $\ell^\infty$ and applications},
   journal={J. Anal. Math.},
   volume={108},
   date={2009},
   pages={219--276},
   issn={0021-7670},
   review={\MR{2544760}},
   doi={10.1007/s11854-009-0024-1},
}

\bib{Host-Kra2018}{book}{
   author={Host, Bernard},
   author={Kra, Bryna},
   title={Nilpotent structures in ergodic theory},
   series={Mathematical Surveys and Monographs},
   volume={236},
   publisher={American Mathematical Society, Providence, RI},
   date={2018},
   pages={X+427},
   isbn={978-1-4704-4780-9},
   review={\MR{3839640}},
   doi={10.1090/surv/236},
}

\begin{comment}
\bib{Krengel1985}{book}{
   author={Krengel, Ulrich},
   title={Ergodic theorems},
   series={De Gruyter Studies in Mathematics},
   volume={6},
   note={With a supplement by Antoine Brunel},
   publisher={Walter de Gruyter \& Co., Berlin},
   date={1985},
   pages={viii+357},
   isbn={3-11-008478-3},
   review={\MR{797411}},
}
\end{comment}

\bib{Kuipers-Niederreiter1974}{book}{
   author={Kuipers, L.},
   author={Niederreiter, H.},
   title={Uniform distribution of sequences},
   series={Pure and Applied Mathematics},
   publisher={Wiley-Interscience [John Wiley \& Sons], New
   York-London-Sydney},
   date={1974},
   pages={xiv+390},
   review={\MR{0419394}},
}

\begin{comment}
\bib{Leibman2005}{article}{
   author={Leibman, A.},
   title={Pointwise convergence of ergodic averages for polynomial sequences
   of translations on a nilmanifold},
   journal={Ergodic Theory Dynam. Systems},
   volume={25},
   date={2005},
   number={1},
   pages={201--213},
   issn={0143-3857},
   review={\MR{2122919}},
   doi={10.1017/S0143385704000215},
}
\end{comment}

\bib{Lesigne1990}{article}{
   author={Lesigne, E.},
   title={Un th\'{e}or\`eme de disjonction de syst\`emes dynamiques et une
   g\'{e}n\'{e}ralisation du th\'{e}or\`eme ergodique de Wiener-Wintner},
   language={French, with English summary},
   journal={Ergodic Theory Dynam. Systems},
   volume={10},
   date={1990},
   number={3},
   pages={513--521},
   issn={0143-3857},
   review={\MR{1074316}},
   doi={10.1017/S014338570000571X},
}

\begin{comment}
\bib{Lesigne1993}{article}{
   author={Lesigne, E.},
   title={Spectre quasi-discret et th\'{e}or\`eme ergodique de
   Wiener-Wintner pour les polyn\^{o}mes},
   language={French, with English and French summaries},
   journal={Ergodic Theory Dynam. Systems},
   volume={13},
   date={1993},
   number={4},
   pages={767--784},
   issn={0143-3857},
   review={\MR{1257033}},
}
\end{comment}

\begin{comment}
\bib{Postnikov-Pyateckii1957}{article}{
   author={Postnikov, A. G.},
   author={Pyatecki\u{\i}, I. I.},
   title={A Markov-sequence of symbols and a normal continued fraction},
   language={Russian},
   journal={Izv. Akad. Nauk SSSR Ser. Mat.},
   volume={21},
   date={1957},
   pages={729--746},
   issn={0373-2436},
   review={\MR{101857}},
}
\end{comment}
	

\bib{Rudolph1990}{book}{
   author={Rudolph, Daniel J.},
   title={Fundamentals of measurable dynamics},
   series={Oxford Science Publications},
   note={Ergodic theory on Lebesgue spaces},
   publisher={The Clarendon Press, Oxford University Press, New York},
   date={1990},
   pages={x+168},
   isbn={0-19-853572-4},
   review={\MR{1086631}},
}

\bib{Wiener-Wintner1941}{article}{
   author={Wiener, Norbert},
   author={Wintner, Aurel},
   title={Harmonic analysis and ergodic theory},
   journal={Amer. J. Math.},
   volume={63},
   date={1941},
   pages={415--426},
   issn={0002-9327},
   review={\MR{0004098}},
   doi={10.2307/2371534},
}

\bib{Xiao2024}{article}{
   author={Xiao, Rongzhong},
   title={Multilinear Wiener-Wintner type ergodic averages and its
   application},
   journal={Discrete Contin. Dyn. Syst.},
   volume={44},
   date={2024},
   number={2},
   pages={425--446},
   issn={1078-0947},
   review={\MR{4671529}},
   doi={10.3934/dcds.2023109},
}

\bib{Zorin-Kranich2015}{article}{
   author={Pavel Zorin-Kranich},
   title={A uniform nilsequence Wiener-Wintner theorem for bilinear ergodic averages},
 journal = {arXiv:1504.04647, 2015}
}

\begin{comment}
\bib{Ziegler2007}{article}{
   author={Ziegler, Tamar},
   title={Universal characteristic factors and Furstenberg averages},
   journal={J. Amer. Math. Soc.},
   volume={20},
   date={2007},
   number={1},
   pages={53--97},
   issn={0894-0347},
   review={\MR{2257397}},
   doi={10.1090/S0894-0347-06-00532-7},
}
\end{comment}
\end{biblist}
\end{bibdiv}
\end{document}